\documentclass[letterpaper,reqno,12pt]{amsart}
\usepackage[margin=1.2in]{geometry}

\usepackage{amssymb, enumerate, color}
\usepackage{mathrsfs}
\usepackage[all]{xy}
\usepackage{hyperref}
\hypersetup{colorlinks=true, bookmarks=true}
\usepackage{enumitem}
\usepackage{graphicx}  

\numberwithin{equation}{section} 

\theoremstyle{plain}
\newtheorem{thm}{Theorem}[section] 
\newtheorem{cor}[thm]{Corollary}
\newtheorem{prop}[thm]{Proposition}

\newtheorem{lem}[thm]{Lemma}
\newtheorem*{mainthm}{Theorem A}
\newtheorem*{mainthm2}{Theorem B}

\theoremstyle{definition} 
\newtheorem{defn}[thm]{Definition}
\newtheorem{lem-defn}[thm]{Lemma-Definition}
\newtheorem{setting}[thm]{Setting}
\newtheorem{eg}[thm]{Example} 

\newtheorem*{notation}{Notation}

\theoremstyle{remark}
\newtheorem{rem}[thm]{Remark}

\newtheorem*{claim}{Claim}
\newtheorem{cln}{Claim}
\newtheorem*{acknowledgement}{Acknowledgments}

\def\ge{\geqslant}
\def\le{\leqslant}
\def\phi{\varphi}
\def\epsilon{\varepsilon}
\def\tilde{\widetilde}

\def\mapsto{\longmapsto}

\def\mod{\operatorname{\,mod}}

\newcommand{\sO}{\mathcal{O}}

\newcommand{\Q}{\mathbb{Q}} 
\newcommand{\C}{\mathbb{C}} 
\newcommand{\R}{\mathbb{R}} 
\newcommand{\Z}{\mathbb{Z}}

\newcommand{\ba}{\mathfrak{a}}
\newcommand{\bb}{\mathfrak{b}}
\newcommand{\m}{\mathfrak{m}}

\newcommand{\q}{\mathfrak{q}}

\newcommand{\WSh}{\mathrm{WSh}}
\newcommand{\WDiv}{\mathrm{WDiv}}
\newcommand{\CDiv}{\mathrm{Div}}
\newcommand{\Diff}{\mathrm{Diff}}
\newcommand{\ACDiff}{\widetilde{\mathrm{Diff}}}
\newcommand{\adj}{\mathrm{adj}}

\newsavebox{\circlebox}
\savebox{\circlebox}{\fontencoding{OMS}\selectfont\Large\char13}
\newlength{\circleboxwdht}

\def\Hom{\operatorname{Hom}}
\def\Spec{\operatorname{Spec}}

\def\Supp{\operatorname{Supp}}

\def\Div{\operatorname{div}}
\def\Exc{\operatorname{Exc}}

\def\ord{\operatorname{ord}}

\title{Deformations of log terminal and semi log canonical singularities}

\author{Kenta Sato}
\address{Faculty of Mathematics, Kyushu University, 744 Motooka, Nishi-ku, Fukuoka 819-0395, Japan}
\email{ksato@math.kyushu-u.ac.jp}

\author{Shunsuke Takagi}
\address{Graduate School of Mathematical Sciences, University of Tokyo, 3-8-1 Komaba, Meguro-ku, Tokyo 153-8914, Japan}
\email{stakagi@ms.u-tokyo.ac.jp}



\dedicatory{Dedicated to Professor Yujiro Kawamata on the occasion of his seventieth birthday.}

\begin{document}

\begin{abstract}
In this paper, we prove that klt singularities are invariant under deformations if the generic fiber is $\Q$-Gorenstein. 
We also obtain a similar result for slc singularities. 
These are generalizations of results of Esnault-Viehweg and S.~Ishii. 
\end{abstract}

\maketitle
\markboth{K.~SATO and S.~TAKAGI}{DEFORMATIONS OF LOG TERMINAL AND SEMI LOG CANONICAL SINGULARITIES}

\tableofcontents

\section{Introduction}

For the purposes of this introduction, we work over the field $\C$ of complex numbers. 
Kawamata log terminal (klt, for short) and log canonical (lc, for short) singularities are important classes of singularities in the minimal model program. 
Esnault-Viehweg \cite{EV} (resp.~S.~Ishii \cite{ish}, \cite{ish2}) proved that two-dimensional klt (resp.~lc) singularities are invariant under small deformations. Unfortunately, an analogous statement fails in higher dimensions, because the general fibers are not necessarily $\Q$-Gorenstein even if the special fiber is klt or lc. In this paper, we give a generalization of their results, using the theory of non-$\Q$-Gorenstein singularities initiated by de Fernex-Hacon \cite{dFH}. 
Our results are not just a formal generalization to the non-$\Q$-Gorenstein setting, but provide a new interpretation of the results of Esnault-Viehweg and Ishii.  

Let $X$ be a normal variety that is not necessarily $\Q$-Gorenstein. 
de Fernex-Hacon \cite{dFH} defined the pullback $f^*D$ of a (non-$\Q$-Cartier) Weil divisor $D$ on $X$, which is a higher-dimensional analog of Mumford's numerical pullback. 
By using this pullback, two relative canonical divisors $K^+_{Y/X}=K_Y+f^*(-K_X)$ and $K^-_{Y/X}=K_Y-f^*K_X$ are defined for every proper birational morphism $f:Y \to X$ from a normal variety $Y$. 
They coincide if $X$ is $\Q$-Gorenstein, but are different in general. 
We say that $X$ has only valuatively klt singularities (resp.~klt singularities in the sense of de Fernex-Hacon) if every coefficient of the $\R$-Weil divisor $K^+_{Y/X}$ (resp.~$K^-_{Y/X}$) is greater than $-1$ for any $f:Y \to X$.\footnote{Valuatively klt singularities are called lt${}^+$ singularities in \cite{Chi} and \cite{CU}. Klt singularities in the sense of de Fernex-Hacon are called klt type singularities in \cite{BGLM}.} 
These singularities are a natural generalization of classical klt singularities to the non-$\Q$-Gorenstein setting, and being valuatively klt is a weaker condition than being klt in the sense of de Fernex-Hacon, because $K^+_{Y/X} \ge K^-_{Y/X}$. 
Klt singularities in the sense of de Fernex-Hacon are known not to be invariant under small deformations (cf.~\cite{Si}), and therefore, we focus on valuatively klt singularities in this paper. 
Our first main result is the inversion of adjunction for valuatively klt singularities, which states that if a Cartier divisor $D$ on $X$ is valuatively klt, then the pair $(X,D)$ is valuatively plt near $D$. 
Here, valuatively plt pairs are a generalization of plt pairs, defined in terms of $K^+_{Y/X}$ as in the case of valuatively klt singularities (see Definition \ref{+ discrepancy} for its precise definition). 
The proof is based on a characterization of valuatively klt singularities in terms of classical multiplier ideal sheaves.  
As its corollary, we obtain the following result on deformations of valuatively klt singularities. 

\begin{mainthm}[Corollary \ref{general klt fibers}]
Let $(\mathcal{X}, \mathcal{D}) \to T$ be a proper flat family of pairs over a complex variety $T$, where $\mathcal{D}$ is an effective $\Q$-Weil divisor on a normal variety $\mathcal{X}$. 
Suppose that some closed fiber $(\mathcal{X}_{t_0}, \mathcal{D}_{t_0})$ is valuatively klt. 
Then so is a general fiber $(\mathcal{X}_{t}, \mathcal{D}_{t})$. 
In particular, if $(\mathcal{X}_{t}, \mathcal{D}_{t})$ is log $\Q$-Gorenstein, then it is klt. 
\end{mainthm}
Theorem A says that klt singularities deform to klt singularities if the general fibers are $\Q$-Gorenstein. 
Note that the total space $\mathcal{X}$ is not (log) $\Q$-Gorenstein in general, and therefore, the classical inversion of adjunction for klt singularities cannot be applied directly even if $T$ is a smooth curve and the fibers are $\Q$-Gorenstein. 
We also remark that the result of Esnault-Viehweg immediately follows from Theorem A, because valuatively klt singularities and classical klt singularities coincide in dimension two (cf.~Lemma \ref{2-dim val lc}). 

Next we discuss deformations of lc singularities. 
Ishii \cite{ish} proved that isolated lc singularities are invariant under small deformations if the general fibers are $\Q$-Gorenstein. 
The condition of isolated singularities are essential in her proof, and in order to remove this condition, we use the notion of valuatively lc singularities, which are defined in a similar way to the valuatively klt case. 
The second main result of this paper proves that if a Cartier divisor $D$ on $X$ is lc, then the pair $(X, D)$ is valuatively lc. 
For the proof, we introduce new variants of Fujino's non-lc ideal sheaves and show that these ideal sheaves behave well under the restriction to a Cartier divisor. 
Then we employ a strategy similar to the klt case but use the variants of Fujino's non-lc ideal sheaves instead of multiplier ideal sheaves. 
We also prove an analogous result for slc singularities, a generalization of lc singularities to the non-normal setting, under a mild additional assumption (see Theorem \ref{slc cut} for details). 
In this case, we use the theory of AC-divisors to deal with divisors on non-normal varieties. 
Since we do not know a suitable reference, the details of the theory are given in the appendix. 
As a corollary of our second main result, we obtain the following generalization of the result of Ishii. 

\begin{mainthm2}[Corollaries \ref{slc proper} and \ref{general lc fibers}]
Let $(\mathcal{X}, \mathcal{D}) \to T$ be a proper flat family of pairs over a smooth complex curve $T$, where $\mathcal{D}$ is an effective $\Q$-Weil divisor on a normal variety $\mathcal{X}$. 
Suppose that some closed fiber $(\mathcal{X}_{t_0}, \mathcal{D}_{t_0})$ is slc. 
Then a general fiber $(\mathcal{X}_t, \mathcal{D}_t)$ is valuatively lc. 
In particular, if $(\mathcal{X}_t, \mathcal{D}_t)$ is log $\Q$-Gorenstein, then it is lc. 
\end{mainthm2}

When the general fibers are log $\Q$-Gorenstein, Theorem B was independently proved by Koll\'ar \cite[Theorem 5.33]{Kol2}, whose method can be traced back to his joint work \cite[Corollary 5.5]{KSB} with Shepherd-Barron, but the proof heavily depends on the existence of lc modifications. 
We believe that our proof, which uses only the vanishing theorem for lc singularities, is of independent interest. 

\begin{acknowledgement}
The first author was partially supported by JSPS KAKENHI Grant Number 20K14303 and the second author was partially supported by JSPS KAKENHI Grant Numbers 16H02141, 17H02831 and 20H00111. 
They are grateful to Shihoko Ishii and J\'anos Koll\'ar for their helpful comments on a preliminary version of this paper. 
The second author would like to thank Osamu Fujino, Yoshinori Gongyo, Yujiro Kawamata and Keiji Oguiso for valuable conversations. 
\end{acknowledgement}

\begin{notation}
Throughout this paper, all rings are assumed to be commutative and with unit element and all schemes are assumed to be Noetherian and separated. 
\end{notation}

\section{Preliminaries}

This section provides preliminary results needed for the rest of the paper. 

\subsection{Singularities in MMP}
In this subsection, we recall the definition and basic properties of singularities in minimal model program (or MMP for short). 

Throughout this subsection, unless otherwise stated, $X$ denotes an excellent normal integral $\Q$-scheme with a dualizing complex $\omega_X^\bullet$. 
The \textit{canonical sheaf} $\omega_X$ associated to $\omega_X^{\bullet}$ is the coherent $\sO_X$-module defined as the first nonzero cohomology of $\omega_X^\bullet$.
A \textit{canonical divisor} of $X$ associated to $\omega_X^\bullet$ is any Weil divisor $K_X$ on $X$ such that $\sO_X(K_X) \cong \omega_X$. 
We fix a canonical divisor $K_X$ of $X$ associated to $\omega_X^\bullet$, and given a proper birational morphism $\pi:Y \to X$ from a normal integral scheme $Y$, we always choose a canonical divisor $K_Y$ of $Y$ that is associated to $\pi^! \omega_X^{\bullet}$ and coincides with $K_X$ outside the exceptional locus $\mathrm{Exc}(f)$ of $f$. 

\begin{defn}
A proper birational morphism $f: Y \to X$ from a regular integral scheme $Y$ is said to be a \textit{resolution of singularities} of $X$. 
When $\Delta$ is a $\Q$-Weil divisor on $X$ and $\ba \subseteq \sO_X$ is a nonzero coherent ideal sheaf,  
a resolution $f:Y \to X$ is said to be a \textit{log resolution} of $(X, \Delta, \ba)$ if $\ba \sO_Y = \sO_Y(-F)$ is invertible 
and if the union of the exceptional locus $\mathrm{Exc}(f)$ of $f$, the support of $F$ and the strict transform $f^{-1}_*\Delta$ of $\Delta$ is a simple normal crossing divisor.  
Log resolutions exist for quasi-excellent $\Q$-schemes (see \cite{Tem}). 
\end{defn}

First, we recall the definition of singularities in MMP. 
\begin{defn}\label{mmp}
Suppose that 
$\Delta$ is an effective $\Q$-Weil divisor on $X$ such that $K_X+\Delta$ is $\Q$-Cartier, $\ba \subseteq \sO_X$ is a nonzero coherent ideal sheaf and $\lambda > 0$ is a real number. 
\begin{enumerate}[label=(\roman*)]
\item 
Given a proper birational morphism $f:Y \to X$ from a normal integral scheme $Y$, we define the $\Q$-Weil divisor $\Delta_Y$ on $Y$ as   
\[
\Delta_Y : = f^*(K_X+\Delta)-K_Y.
\]
When $\ba \sO_Y=\sO_Y(-F)$ is invertible, 
the \textit{discrepancy} $a_{E}(X, \Delta, \ba^\lambda)$ of the triple $(X, \Delta, \ba^\lambda)$ with respect to a prime divisor $E$ on $Y$ is defined as the coefficient of $E$ in $-(\Delta_Y+\lambda F)$. 
\item The triple $(X, \Delta, \ba^\lambda)$ is said to be \textit{log canonical} (or \textit{lc} for short) if $a_E(X, \Delta, \ba^{\lambda}) \ge -1$ for every proper birational morphism $f:Y \to X$ from a normal integral scheme $Y$ with $\ba \sO_Y$ invertible and for every prime divisor $E$ on $Y$. 
\end{enumerate}
\end{defn}

\begin{defn}\label{defn plt}
Suppose that $\Delta$ is an effective $\Q$-Weil divisor on $X$, $\ba \subseteq \sO_X$ is a nonzero coherent ideal sheaf and $\lambda > 0$ is a real number.
Let $D$ be a reduced Weil divisor on $X$ which has no common components with $\Delta$ and none of whose generic points lies in the zero locus of $\ba$. 
Assume in addition that $K_X+\Delta+D$ is $\Q$-Cartier.
\begin{enumerate}[label=(\roman*)]
\item The triple $(X,\Delta+D, \ba^\lambda)$ is said to be \textit{purely log terminal} (or \textit{plt} for short) along $D$ if $a_E(X, \Delta+D, \ba^{\lambda}) > -1$ for every proper birational morphism $f:Y \to X$ from a normal integral scheme $Y$ with $\ba \sO_Y$ invertible and for every prime divisor $E$ on $Y$ that is not an irreducible component of the strict transform $f^{-1}_*D$  of $D$. 
\item The \textit{adjoint ideal sheaf} $\adj_D(X, \Delta+D, \ba^{\lambda})$ associated to $(X,\Delta+D, \ba^\lambda)$ along $D$ is defined as
\[
\adj_D(X, \Delta+D, \ba^\lambda):= \bigcap_{f: Y \to X} f_* \sO_Y(-\lfloor (\Delta+D)_Y-f^{-1}_*D + \lambda F \rfloor),
\]
where $f: Y \to X$ runs through all proper birational morphisms from a normal integral scheme $Y$ with $\ba \sO_Y= \sO_Y(-F)$ invertible. 
\item Assume that $K_X+\Delta$ is $\Q$-Cartier. 
The triple $(X, \Delta, \ba^\lambda)$ is said to be \textit{Kawamata log terminal} (or \textit{klt} for short) if it is plt along the zero divisor. 
The adjoint ideal sheaf $\adj_0(X, \Delta, \ba^\lambda)$ is called the \textit{multiplier ideal sheaf} associated to $(X, \Delta, \ba^\lambda)$ and is denoted by $\mathcal{J}(X,\Delta,\ba^\lambda)$.
\end{enumerate}
\end{defn}

\begin{rem}[\textup{cf.~\cite[9.3.E]{Laz}}]
Let $(X,\Delta, \ba^\lambda)$ and $D$ be as in Definition \ref{defn plt}.
\begin{enumerate}[label=(\roman*)]
\item $(X, \Delta+D, \ba^\lambda)$ is plt along $D$ if and only if $\adj_D(X, \Delta+D, \ba^\lambda)=\sO_X$.
\item If $f: Y \to X$ is a log resolution of $(X, \Delta+D, \ba)$ separating the components of $D$, then 
\[
\adj_D(X, \Delta+D, \ba^\lambda) = f_* \sO_Y(- \lfloor (\Delta+D)_Y -f^{-1}_*D + \lambda F \rfloor). 
\]
\end{enumerate}
\end{rem}

Next, we introduce a generalization of the singularities in Definitions \ref{mmp} and \ref{defn plt} to the non-$\Q$-Gorenstein setting.  
\begin{defn}[\textup{\cite[Section 2]{dFH}}]\label{dFH valuation}
Suppose that 
$f : Y \to X$ is a proper birational morphism from a normal integral scheme $Y$ and $E$ is a prime divisor on $Y$. 
The discrete valuation associated to $E$ is denoted by $\ord_E$. 
\begin{enumerate}[label=(\roman*)]
\item  The \textit{natural valuation} $\ord_E^{\natural}(D)$ along $\ord_E$ of a Weil divisor $D$ on $X$ is defined as the integer 
\[
\ord_E^{\natural}(D) : = \ord_E( \sO_X(-D)).
\]
The \textit{natural pullback} of $D$ on $Y$ is the Weil divisor
\[
f^{\natural}(D) : = \sum_E \ord_E^{\natural}(D)E,
\]
where $E$ runs through all prime divisors on $Y$.

\item The \textit{valuation} $\ord_E(D)$ along $\ord_E$ of a $\Q$-Weil divisor $D$ on $X$ is defined as the real number 
\[
\ord_E(D) : = \lim_{m \to \infty} \frac{\ord_E^{\natural}(mD)}{m} = \inf_{m \ge 1} \frac{\ord_E^{\natural}(mD)}{m},
\]
where the limit is taken over all integers $m \ge 1$ such that $m D$ is an integral Weil divisor.
This limit always exists by \cite[Lemma 1.4]{Mus}.
The \textit{pullback} of $D$ on $Y$ is the $\R$-Weil divisor
\[
f^{*}(D) : = \sum_E \ord_E (D)E,
\]
where $E$ runs through all prime divisors on $Y$.
\end{enumerate}
\end{defn}

\begin{defn}\label{+ discrepancy}
Suppose that 
$\Delta$ is an effective $\Q$-Weil divisor on $X$, $\ba \subseteq \sO_X$ is a nonzero coherent ideal sheaf and $\lambda > 0$ is a real number. 
Let $m>0$ be an integer such that $m \Delta$ is an integral Weil divisor and $D$ be a reduced Weil divisor on $X$ which has no common components with $\Delta$ and none of whose generic points lies in the zero locus of $\ba$. 
\begin{enumerate}[label=(\roman*)]
\item Let $f: Y \to X$ be a proper birational morphism from a normal integral scheme $Y$ with $\ba \sO_Y=\sO_Y(-F)$ invertible and let $E$ be a prime divisor on $Y$. 
The \textit{$m$-th limiting discrepancy} of $(X,\Delta, \ba^\lambda)$ is defined as 
\begin{align*}
a_{m, E}^{+} (X, \Delta, \ba^\lambda) &= \ord_E(K_Y -\lambda F) + \frac{\ord_E^{\natural}(-m (K_X+\Delta))}{m}.
\end{align*}
The \textit{discrepancy} of $(X,\Delta, \ba^\lambda)$ is defined as 
\begin{align*}
a_{E}^{+} (X, \Delta, \ba^\lambda) &= \ord_E(K_Y -\lambda F) + \ord_E(-(K_X+ \Delta)) \\
&= \lim_{n \to \infty} a_{n, E}^{+}(X,\Delta, \ba^\lambda)\\
&= \inf_{n} a_{n, E}^{+}(X,\Delta, \ba^\lambda),
\end{align*}
where the limit and the infimum are taken over all integer $n \ge 1$ such that $n\Delta$ is an integral Weil divisor.
When $\ba=\sO_X$, we simply write $a_{m, E}^{+} (X, \Delta)$ (resp.~$a_{E}^{+} (X, \Delta)$) instead of $a_{m, E}^{+} (X, \Delta, \ba^\lambda)$ (resp.~$a_{E}^{+} (X, \Delta, \ba^\lambda)$). 

\item (\cite{Yo}) We say that $(X, \Delta, \ba^\lambda)$ is \textit{valuatively lc} (resp.~\textit{$m$-weakly valuatively lc}) if $a_{E}^{+}(X, \Delta, \ba^\lambda) \ge -1$ (resp.~$a_{m,E}^{+}(X, \Delta, \ba^\lambda) \ge -1$) for every proper birational morphism $f:Y \to X$ from a normal integral scheme $Y$ with $\ba \sO_Y$ invertible and for every prime divisor $E$ on $Y$. 

\item We say that $(X, \Delta+D, \ba^\lambda)$ is \textit{valuatively plt} (resp.~\textit{$m$-weakly valuatively plt}) along $D$ if $a_{E}^{+}(X, \Delta+D, \ba^\lambda) > -1$ (resp.~$a_{m,E}^{+}(X, \Delta+D, \ba^\lambda) > -1$) for every proper birational morphism $f:Y \to X$ from a normal integral scheme $Y$ with $\ba \sO_Y$ invertible and for every prime divisor $E$ on $Y$ that is not an irreducible component of $f^{-1}_*D$. 

\item We say that $(X, \Delta, \ba^\lambda)$ is \textit{valuatively klt}\footnote{Valuatively klt singularities are called lt$^+$ singularities in \cite{Chi} and \cite{CU}.} (resp.~\textit{$m$-weakly valuatively klt}) if it is valuatively plt (resp.$m$-weakly valuatively plt) along the zero divisor.
\end{enumerate}
\end{defn}

\begin{rem}\label{valuative implication}
Let $(X,\Delta+D, \ba^{\lambda})$ be as in Definition \ref{+ discrepancy}. 
If $(X, \Delta+D, \ba^{\lambda})$ is valuatively plt along $D$, then $(X, \Delta+D, \ba^{\lambda})$ is valuatively lc and $(X, \Delta, \ba^{\lambda})$ is valuatively klt. 
This follows from the fact that if $D_i$ is an irreducible component of $D$, then 
\[
a^+_{f^{-1}_*D_i}(X, \Delta+D)=a^+_{D_i}(X, \Delta+D)=-1, \quad a^+_{f^{-1}_*D_i}(X, \Delta)=a^+_{D_i}(X, \Delta)=0
\]
 for every proper birational morphism $f:Y \to X$ from a normal integral scheme $Y$.  
\end{rem}

\begin{rem}\label{remark on val lc}
Let $(X,\Delta, \ba^\lambda)$ be as in Definition \ref{+ discrepancy}.
If $K_X+\Delta$ is $\Q$-Cartier, then $(X, \Delta, \ba^{\lambda})$ is lc if and only if it is valuatively lc.
Similarly, if $K_X+\Delta+D$ is $\Q$-Cartier, then the following three conditions are equivalent to each other:
\begin{enumerate}[label=\textup{(\alph*)}]
\item $(X, \Delta+D, \ba^\lambda)$ is plt along $D$,
\item $(X, \Delta+D, \ba^\lambda)$ is valuatively plt along $D$, and 
\item $(X, \Delta+D, \ba^\lambda)$ is $m$-weakly valuatively plt along $D$ for every integer $m \ge 1$ such that $m\Delta$ an integral Weil divisor. 
\end{enumerate}
\end{rem}

\begin{lem}\label{single resolution}
Suppose that $(X,\Delta, \ba^\lambda)$, $m$ and $D$ are as in Definition \ref{+ discrepancy}.
Let $f: Y \to X$ be a log resolution of $(X, \Delta+D, \ba)$ separating the components of $D$. 
\begin{enumerate}[label=$(\arabic*)$]
\item The triple $(X, \Delta, \ba^\lambda)$ is $m$-weakly valuatively lc $($resp.~valuatively lc$)$ if and only if $a_{m,E}^+(X, \Delta, \ba^{\lambda}) \ge -1$ $($resp.~$a_E^+(X, \Delta , \ba^\lambda) \ge -1$$)$ for every prime divisor $E$ on $Y$.
\item $(X, \Delta+ D, \ba^\lambda)$ is $m$-weakly valuatively plt $($resp.~valuatively plt$)$ along $D$ if and only if $a_{m,E}^+(X, \Delta+D, \ba^{\lambda}) > -1$ $($resp.~$a_E^+(X, \Delta +D, \ba^\lambda)>-1$$)$ for every prime divisor $E$ on $Y$ that is not an irreducible component of $f^{-1}_*D$.
\end{enumerate}
\end{lem}

\begin{proof}
The assertion follows from \cite[Lemma 2.7]{dFH} and \cite[Remark 2.13]{dFH}.
\end{proof}

\begin{prop}\label{singularity at a point}
Suppose that $(X,\Delta, \ba^\lambda)$, $m$ and $D$ are as in Definition \ref{+ discrepancy}.
Let $x \in X$ be a point and let $\Delta_x$ and $D_x$ denote the flat pullbacks of $\Delta$ and $D$ by the canonical morphism $\Spec \sO_{X,x} \to X$, respectively.
\begin{enumerate}[label=$(\arabic*)$]
\item $(X, \Delta, \ba)$ is \textit{valuatively lc at $x$}, that is, $(\Spec \sO_{X,x}, \Delta_x, \ba \sO_{X, x})$ is valuatively lc if and only if there exists an open neighborhood $U \subseteq X$ of $x$ such that $(U, \Delta|_U, \ba|_U)$ is valuatively lc. 
\item $(X, \Delta+D, \ba)$ is \textit{valuatively plt along $D$ at $x$}, that is, $(\Spec \sO_{X,x}, \Delta_x+D_x, \ba \sO_{X, x})$ is valuatively plt along $D_x$ if and only if there exists an open neighborhood $U \subseteq X$ of $x$ such that $(U, \Delta|_U +D|_U, \ba|_U)$ is valuatively plt along $D|_U$.
\end{enumerate}
\end{prop}

\begin{proof}
This is an immediate application of Lemma \ref{single resolution}.
\end{proof}

\begin{prop}\label{weak sense vs strong sense}
Let $(X,\Delta, \ba^\lambda)$ and $D$ be as in Definition \ref{+ discrepancy}.
Suppose that $I \subseteq \sO_X$ is an coherent ideal sheaf whose zero locus does not contain any generic points of $D$ but contains the locus where $K_X+\Delta+D$ is not $\Q$-Cartier.
Then $(X,\Delta+D, \ba^{\lambda})$ is valuatively plt along $D$ if and only if there exists a real number $\epsilon>0$ such that $(X,\Delta+D, \ba^{\lambda} I^\epsilon)$ is $m$-weakly valuatively plt along $D$ for every integer $m \ge 1$ with $m\Delta$ an integral Weil divisor. 
\end{prop}

\begin{proof}
Take a log resolution $f: Y \to X$ of $(X, \Delta+D, \ba I)$ separating the components of $D$, and let $F$ and $G$ be Cartier divisors on $Y$ such that $\sO_Y(-F)= \ba \sO_Y$ and $\sO_Y(-G)=I \sO_Y$. 
For all integer $m \ge 1$ such that $m\Delta $ is an integral Weil divisor, we define the $\R$-Weil divisors $(\Delta+D)_{Y}^+$ and $(\Delta+D)_{m,Y}^+$ on $Y$ as 
\begin{align*}
(\Delta+D)_{Y}^+  &: = -f^{*}(-(K_X+\Delta+D)) - K_Y=-\sum_E a_E^+(X, \Delta+D)E,\\
(\Delta+D)_{m,Y}^+ &: = -\frac{f^{\natural}(-m(K_X+\Delta+D))}{m} - K_Y=-\sum_E a_{m,E}^+(X, \Delta+D)E,
\end{align*}
where $E$ runs through all prime divisors on $Y$. 

To prove the ``only if" part, it suffices to show by Lemma \ref{single resolution} that there exists a real number $\epsilon>0$ such that 
\[
\ord_E((\Delta+D)_{m,Y}^{+} -f^{-1}_*D + \lambda F + \epsilon G ) <1
\] 
for every integer $m \ge 1$ with $m\Delta$ an integral Weil divisor and for every prime divisor $E$ on $Y$.
Since $(X, \Delta+D, \ba^\lambda)$ is valuatively plt along $D$, 
\[
\ord_E((\Delta+D)_Y^{+} -f^{-1}_*D+ \lambda F )<1
\] 
for every prime divisor $E$ on $Y$.
Therefore, there exists $\epsilon > 0$ such that 
\[
\ord_E((\Delta+D)_Y^{+} -f^{-1}_*D + \lambda F + \epsilon G)<1
\]
for all prime divisors $E$ on $Y$.
Then we have 
\[
\ord_E((\Delta+D)_{m,Y}^{+} -f^{-1}_*D + \lambda F + \epsilon G ) \le \ord_E((\Delta+D)_Y^{+} -f^{-1}_*D+ \lambda F + \epsilon G)<1. 
\]

For the ``if" part, we fix a prime divisor $E$ on $Y$.
It is enough to show by Lemma \ref{single resolution} that 
\[
\ord_E((\Delta+D)_{Y}^{+} -f^{-1}_*D  + \lambda F ) <1.
\] 
If $K_X+\Delta+D$ is $\Q$-Cartier at the center of $E$, then this inequality follows from Remark \ref{remark on val lc}.
Therefore, we may assume that $K_X+\Delta+D$ is not $\Q$-Cartier at the center of $E$.
Then by the definition of $I$, the center of $E$ is contained in the zero locus of $I$, which implies that $\ord_E(G) >0$. 
Since $(X, \Delta+D, \ba^\lambda I^\epsilon)$ is $m$-weakly valuatively plt along $D$ for all $m \ge 1$ such that $m\Delta$ is an integral Weil divisor, 
\[
\ord_E((\Delta+D)_{m,Y}^{+} -f^{-1}_*D + \lambda F ) <1 - \epsilon \ord_E(G). 
\] 
Taking the supremum over all such $m$, we have 
\[
\ord_E((\Delta+D)_{Y}^{+}  -f^{-1}_*D + \lambda F )  \le 1 - \epsilon \ord_E(G) < 1. 
\]
\end{proof}

\begin{lem}[\textup{cf.~\cite[Proposition 4.11 (2)]{KM}}]\label{2-dim val lc}
Let $(X,\Delta, \ba^\lambda)$ be as in Definition \ref{+ discrepancy}. 
If $(X, \Delta, \ba^\lambda)$ is valuatively lc at a point $x \in X$ with $\dim \sO_{X,x} \le 2$, then $K_X+\Delta$ is $\Q$-Cartier at $x$, and therefore, $(X,\Delta, \ba^\lambda)$ is lc at $x$.
\end{lem}

\begin{proof}
Since the pullback of a $\Q$-Weil divisor on a surface, defined in Definition \ref{dFH valuation}, coincides with Mumford's numerical pullback, the pair $(X,\Delta)$ is  numerically lc at $x$ (see \cite[\S 4.1]{KM} for the definition of numerically lc pairs). 
The assertion then follows from \cite[Proposition 4.11 (2)]{KM}.  
\end{proof}

The log canonicity can be generalized for non-connected schemes in a natural way. 
\begin{defn}
Let $X$ be an excellent normal (not necessarily connected) $\Q$-scheme with a dualizing complex $\omega_X^{\bullet}$, $\Delta$ be an effective $\Q$-Weil divisor on $X$, $\lambda>0$ be a real number and $\ba \subseteq \sO_X$ be a coherent ideal that is nonzero at any generic points of $X$.
Let $X : = \coprod_i X_i$ be the decomposition of $X$ into connected components. 
We say that $(X, \Delta, \ba^{\lambda})$ is \textit{lc} (resp.~\textit{valuatively lc}) if so is $(X_i, \Delta|_{X_i}, \ba|_{X_i}^{\lambda})$ for all $i$.
\end{defn}

\subsection{Semi log canonical singularities}
Throughout this subsection, we assume that $X$ is an excellent reduced scheme satisfying Serre's condition $(S_2)$.
Let $\mathcal{K}_X$ denote the sheaf of total quotients of $X$.

We define the abelian groups $\mathrm{WDiv}^*(X)$ and $\WDiv_\Q^*(X)$ as 
\begin{align*}
\mathrm{WDiv}^*(X) &: = \bigoplus_{E} \Z E, \\
\mathrm{WDiv}_\Q^*(X) & : = \WDiv^*(X) \otimes_\Z \Q = \bigoplus_{E} \Q E ,
\end{align*}
where $E$ runs through all prime divisors on $X$ whose generic points are regular points of $X$. 
Similarly, let $\mathrm{Div}^*(X)$ be the subgroup of $\CDiv(X) = \Gamma(X, \mathcal{K}_X^*/\sO_X^*)$ defined as
\begin{multline*}
\mathrm{Div}^*(X) : = \{ C \in \mathrm{Div}(X) \mid C_x = 1 \mod \sO_{X,x}^*\  \textup{for every codimension one}  \\ \textup{singular point } x \in X \}
\end{multline*}
It follows from \cite[Theorem 11.5 (ii)]{Mat} 
that the canonical map 
\[
\mathrm{Div}^*(X) \to \mathrm{WDiv}^*(X)
\]
is injective.\footnote{The $(R_1)$ condition is assumed in loc.~cit., but this assumption is unnecessary for the injectivity.}

Let $D$ be a Weil divisor contained in $\WDiv^*(X)$.
Since the support of $D$ contains no codimension one singular points of $X$, there exists an open subset $U \subseteq X$ containing all codimension one points of $X$ such that the restriction $D|_U \in \WDiv^*(U)$ of $D$ is Cartier, that is, there exists a (unique) Cartier divisor $E_U$ on $U$ contained in $\CDiv^*(U)$ such that the Weil divisor defined by $E_U$ coincides with $D|_U$.
Then we define the subsheaf $\sO_X(D)$ of $\mathcal{K}_X$ as the pushforward $i_* \sO_U(E_U)$ of the invertible subsheaf $\sO_U(E_U) \subseteq \mathcal{K}_U$ by the open immersion $i : U \hookrightarrow X$. 

\begin{lem}
The quasi-coherent $\sO_X$-module $\sO_X(D)$ is coherent, reflexive and independent of the choice of $U$.
\end{lem}

\begin{proof}
We write
\[
D= \sum_{i=1}^n a_i E_i - \sum_j^m b_j E_j,
\]
where $E_i$ and $E_j$ are prime divisors on $X$ whose generic points are regular points of $X$ and $a_i$ and $b_j$ are positive integers.
Let $\mathcal{G}$ be the coherent sheaf 
\[
\mathcal{H}om_{X}\left( \bigotimes_i  \mathcal{I}_{E_i}^{\otimes a_i}, \Big(\bigotimes_j  \mathcal{I}_{E_j}^{\otimes b_j}\Big)^{**}\right),
\]
where $\mathcal{I}_{E_i}$ (resp.~$\mathcal{I}_{E_j}$) is the ideal sheaf of $E_i$ (resp.~$E_j$) and $(-)^{**}$ denotes the reflexive hull.
Note by \cite[Corollary 2.9]{Sch10} that $\mathcal{G}$ is reflexive.\footnote{$X$ is assumed to be irreducible in loc.~cit., but this assumption is unnecessary.}

Let $j: V \hookrightarrow U$ be the open immersion from an open subset $V \subseteq U$ containing all codimension one points of $X$ such that $E_i|_V$ (resp.~$E_j|_V$) is Cartier for all $i$ (resp.~$j$).
Since $\mathcal{G}|_V \cong \sO_V(E_U |_V)$, it follows from Lemma \ref{S2} that 
\[
\mathcal{G} \cong (i \circ j)_* (\mathcal{G}|_V)  \cong i_* j_* \sO_V(E_U|_V)  \cong i_* \sO_U(E_U) =  \sO_X(D). 
\]
\end{proof}

\begin{lem}\label{S2}
Let $X$ be a Noetherian reduced $(S_2)$ scheme and $\mathcal{F}$ be a coherent sheaf.
Then the following conditions are equivalent to each other. 
\begin{enumerate}[label=$(\arabic*)$]
\item $\mathcal{F}$ is reflexive.
\item $\mathcal{F}$ satisfies $(S_2)$ and $\mathcal{F}$ is reflexive in codimension one, that is, $\mathcal{F}_x$ is a reflexive $\sO_{X,x}$-module for each codimension one point $x \in X$. 
\item $\mathcal{F}$ is reflexive in codimension one and the natural map $\mathcal{F} \to i_*i^* \mathcal{F}$ is an isomorphism for every open subscheme $i : U \hookrightarrow X$ with $\mathrm{Codim}(X \setminus U, X) \ge 2$.  
\item There exists a reflexive sheaf $\mathcal{G}$ on an open subscheme $i: U \hookrightarrow X$ with $\mathrm{Codim}(X \setminus U, X) \ge 2$ such that $\mathcal{F} \cong i_* \mathcal{G}$.
\end{enumerate}
\end{lem}

\begin{proof}
The proof is very similar to the argument in \cite[Section 2]{Sch10}.
\end{proof}

Let $\omega_X^{\bullet}$ be a dualizing complex of $X$ and $\omega_X$ be the canonical sheaf associated to $\omega_X^{\bullet}$, that is, the coherent $\sO_X$-module defined as the first nonzero cohomology of $\omega_X^{\bullet}$. 
A \emph{canonical divisor} on $X$ associated to $\omega_X^{\bullet}$ is a Weil divisor $K_X$ contained in $\WDiv^*(X)$ such that $\sO_X(K_X) \cong \omega_X$ as $\sO_X$-modules.
The following proposition gives sufficient conditions for $X$ to admit a canonical divisor.

\begin{prop}\label{canonical divisor exists}
Let $(\Lambda, \m, k)$ be a Noetherian local ring with $k$ infinite and $A$ be an excellent $\Lambda$-algebra.
Suppose that $X$ is a reduced, $(S_2)$, $(G_1)$ and quasi-projective $A$-scheme with a dualizing complex $\omega_X^{\bullet}$.
Then $X$ admits a canonical divisor associated to $\omega_X^{\bullet}$ if one of the following conditions hold.
\begin{enumerate}[label=\textup{(\roman*)}]
\item There exists a finite morphism $f: X \to Y$ to an excellent reduced $(S_2)$ and $(G_1)$ scheme $Y$ with the following conditions:
\begin{enumerate}[label=\textup{(\alph*)}]
\item $Y$ admits a dualizing complex $\omega_Y^{\bullet}$ such that $f^! \omega_Y^{\bullet} \cong \omega_X^{\bullet}$,
\item $Y$ admits a canonical divisor associated to $\omega_Y^{\bullet}$, and
\item the codimension of $f(\eta) \in Y$ is constant for all generic points $\eta$ of $X$.
\end{enumerate}
\item $X$ is irreducible.
\item $X$ is connected and biequidimensional (see Definition \ref{biequidim} for the definition of biequidimensional schemes). 
\end{enumerate}
\end{prop}

\begin{proof}
It follows from Lemmas \ref{canonical AC divisor2} and \ref{Moving}.
\end{proof}

Let $\nu : X^n \to X$ be the normalization of $X$ and $C \in \mathrm{WDiv}(X^n)$ be the conductor divisor of $\nu$ on $X^n$, that is, an effective Weil divisor on $X^n$ satisfying that 
\[
\sO_{X^n}(-C) = \nu^{-1} (\mathcal{H}om_X(\nu_* \sO_{X^n}, \sO_X)) \subseteq \sO_{X^n}.
\]
If $X$ admits a canonical divisor $K_X$ associated to $\omega_X^{\bullet}$, then it follows from \cite[Subsection 5.1]{Kol} that the Weil divisor $\nu^* K_X-C$
on $X^n$ is a canonical divisor associated to the dualizing complex $\nu^{!} \omega_X^{\bullet}$.

\begin{defn}\label{defn slc}
Let $X$ be an excellent reduced $(S_2)$ and $(G_1)$ $\Q$-scheme admitting a canonical divisor $K_X \in \WDiv^*(X)$ associated to a dualizing complex $\omega_X^{\bullet}$. 
Suppose that $\Delta \in \WDiv^*_{\Q}(X)$ is an effective $\Q$-Weil divisor, $\ba \subseteq \sO_X$ is a coherent ideal sheaf that is nonzero at any generic points of $X$ and $\lambda >0$ is a real number.
\begin{enumerate}
\item The triple $(X, \Delta, \ba^\lambda)$ is said to be \textit{semi log canonical} (or \textit{slc} for short) if $K_X+\Delta$ is $\Q$-Cartier and $(X^n, \nu^* \Delta + C, (\ba \sO_{X^n})^{\lambda})$ is lc.
\item The triple $(X, \Delta, \ba^\lambda)$ is said to be \textit{valuatively slc} if $(X^n, \nu^* \Delta + C, (\ba \sO_{X^n})^{\lambda})$ is valuatively lc.

\end{enumerate}
\end{defn}

\begin{rem}\label{remark on slc} 
(1) There exists an example of a $2$-dimensional non-$\Q$-Gorenstein valuatively slc scheme (see \cite[Example 5.16]{Kol}).

(2) Let $(X,\Delta, \ba^\lambda)$ be as in Definition \ref{defn slc} and assume in addition that $X$ is a $\Q$-scheme and $x \in X$ is a point.  
It then follows from Proposition \ref{singularity at a point} that $(X, \Delta, \ba^\lambda)$ is \textit{valuatively slc at $x$}, that is, the induced triple $(\Spec \sO_{X,x}, \Delta_x, (\ba \sO_{X,x})^\lambda)$ is valuatively slc if and only if $(U, \Delta|_U, \ba|_U^\lambda)$ is valuatively slc for an open neighborhood $U \subseteq X$ of $x$.
\end{rem}

\subsection{Different}\label{Subsection Diff}

In this subsection, we recall the definition and basic properties of the different of a $\Q$-Weil divisor. 
The detailed proofs are given in Appendix \ref{appendix} (see also \cite[Subsection 4.1]{Kol}). 

Throughout this subsection, we fix an excellent scheme $S$ admitting a dualizing complex $\omega_S^{\bullet}$, 
every scheme is assumed to be separated and of finite type over $S$ and every morphism is assumed to be an $S$-morphism.
Moreover, given a scheme $X$, we always choose 
$\omega_X^{\bullet} : = \pi_X^{!} \omega_S^{\bullet}$ as a dualizing complex of $X$, 
where $\pi_X: X \to S$ is the structure morphism, and $\omega_X$ always denotes the canonical sheaf associated to $\omega_X^{\bullet}$.

\begin{setting}\label{setup for Diff}
Let $(Y, W, W', i, \mu, f, \Delta)$ be a tuple satisfying the following conditions. 
\begin{enumerate}
\item $Y$ is an excellent reduced $(S_2)$ and $(G_1)$ scheme over $S$ admitting a canonical divisor $K_Y \in \WDiv^*(Y)$ associated to $\omega_Y^{\bullet}:= \pi_Y^! \omega_S^{\bullet}$, where $\pi_Y:Y \to S$ is the structure morphism. 
\item $i: W \hookrightarrow Y$ is the closed immersion from a reduced closed subscheme $W$ whose generic points are codimension one regular points of $Y$. In particular, $W \in \WDiv^*(Y)$. 
\item $\mu: W' \to W$ is a finite birational morphism from a reduced $(S_2)$ and $(G_1)$ scheme $W'$ and $f: = i \circ \mu : W' \to Y$ is the composite of $i$ and $\mu$. 
\[
\xymatrix{
Y & \\
W \ar^-{i}[u] & W' \ar^-{\mu}[l] \ar_-{f}[ul]
}\]
\item $\Delta \in \WDiv^*_{\Q}(Y)$ is a $\Q$-Weil divisor on $Y$ such that the support of $\Delta$ has no common components with $W$ and $K_Y+ \Delta+W \in \WDiv_\Q^*(Y)$ is $\Q$-Cartier at every codimension one point $w$ of $W$.
\item For each codimension one singular point $w'$ of $W'$, there exists an open neighborhood $U \subseteq Y$ of $f(w') \in Y$ such that $\Delta|_U=0$ and $W|_U$ is Cartier, that is, $W|_U$ is contained in the image of the natural injection $\CDiv^*(U) \to \WDiv^*(U)$.
\end{enumerate}
\end{setting}

Let $(Y, W, W', i, \mu, f, \Delta)$ be as in Setting \ref{setup for Diff}.
Then the $\Q$-Weil divisor $\Diff_{W'}(\Delta) \in \WDiv^*_\Q(W')$ is defined as in \cite[Subsection 4.1]{Kol} and is called the \emph{different} of $\Delta$ on $W'$.
The reader is referred to Lemma-Definition \ref{Definition of Diff}  for details. 

\begin{rem}
The condition (5) in Setting \ref{setup for Diff} is not essential.
In Subsection \ref{Subsection AC-Diff}, we remove this condition by formulating the different $\Diff_{W'}(\Delta)$ in terms of AC-divisors. 
\end{rem}

\begin{lem}\label{Diff vs Cond1}
Let $(Y, W, W', i, \mu, f, \Delta)$ be as in Setting \ref{setup for Diff} and $\pi: W^n = (W')^n \to W'$ be the normalization of $W'$.
Then
\[
\Diff_{W^n}(\Delta) = \pi^* \Diff_{W'}(\Delta) + C_{W'},
\]
where $C_{W'}$ denotes the conductor divisor of $\pi$ on $W^n=(W')^n$.
\end{lem}

\begin{proof}
This is a special case of Lemma \ref{AC-Diff vs Cond1}.
\end{proof}

\begin{lem}\label{Diff vs Cond2}
Let $(Y, W, W', i, \mu, f, \Delta)$ be as in Setting \ref{setup for Diff} such that $f: W' \to Y$ factors through the normalization $\nu: Y^n \to Y$ of $Y$.
We further assume that $Y$ is normal at $f(w') \in Y$ for every codimension one singular point $w' \in W'$.
Then 
\[
\Diff_{W'}(\Delta) = \Diff_{W'}(\nu^* \Delta +C_Y),
\]
where $C_Y$ denotes the conductor divisor of $\nu$ on $Y^n$.
\end{lem}

\begin{proof}
We write $\mathcal{A}' : =(Y^n, W'', W', j, \pi, g , \nu^*\Delta+C_Y)$, where $g: W' \to Y^n$ is the morphism induced by $f$, $W'' \subseteq Y^n$ is the reduced image of $g$, 
and $j$, $\pi$ and $\rho$ are natural morphisms such that the following diagram commutes: 
\[
\xymatrix{
 Y  & Y^n \ar^-{\nu}[l]& \\
W \ar@{^{(}->}^-{i}[u]  & W'' \ar@{^{(}->}^-{j}[u] \ar^-{\rho}[l] & W'. \ar^-{\pi}[l] \ar@/^18pt/[ll]^{\mu} \ar^-{g}[ul] \ar@/_30pt/[ull]_{f}
}\]
It is clear that $\mathcal{A}'$ satisfies the conditions (1)--(3) in Setting \ref{setup for Diff}. 
The tuple $\mathcal{A}'$ also satisfies (4), because $\nu$ is an isomorphism over the generic points of $W$, and therefore, $\nu^*W=W''$. 
By the assumption that $Y$ is normal at the image of every codimension one singular point $w'$ of $W'$, the conductor divisor $C_Y$ is trivial near $g(w')$, which implies that $\mathcal{A}'$ satisfies the condition (5) too. 
Then the assertion is a special case of Lemma \ref{AC-Diff vs Cond2}.
\end{proof}

\begin{rem}
We can relax the assumption that $Y$ is normal at the image of any codimension one singular points of $W'$ by using the terminology of AC-divisors.
See Lemma \ref{AC-Diff vs Cond2} for details.
\end{rem}

\begin{lem}\label{restriction of divisor}
Suppose that $Y$ is a scheme satisfying the condition $(1)$ in Setting \ref{setup for Diff} and $i: W \hookrightarrow Y$ be a closed immersion satisfying the condition $(2)$. 
We further assume that $W$ is a Cartier divisor (that is, $W \in \CDiv^*(Y)$) satisfying $(S_2)$ and $(G_1)$. 
Let $\Delta = \sum_i a_i E_i \in \WDiv^*_{\Q}(Y)$ be a $\Q$-Weil divisor on $Y$ whose support contains neither any generic points of $W$ nor any singular codimension one points of $W$. 
\begin{enumerate}[label=$(\arabic*)$]
\item The tuple $(Y, W, W, i, \mathrm{id}_{W},i, \Delta)$ satisfies all the conditions in Setting \ref{setup for Diff}.
\item $($\cite[Proposition 4.5 (4)]{Kol}$)$ Let $\Delta|_W \in \WDiv_\Q^*(W)$ be the restriction 
\[
\Delta|_W : = \sum_i a_i E_i|_W \in \WDiv^*_{\Q}(W)
\] of $\Delta$ to $W$, where $E_i|_W$ denotes the Weil divisor on $W$ corresponding to the scheme theoretic intersection $E_i \cap W$. 
Then 
\[
\Delta|_W = \Diff_W(\Delta).
\]
\end{enumerate}
\end{lem}

\begin{proof}
This is just a reformulation of Lemma \ref{restriction of AC-divisor}.
\end{proof}

\subsection{Deformations}

In this subsection, we recall some basic terminology from the theory of deformations.

\begin{defn}
Let $X$ be an algebraic scheme over a field $k$.  Suppose that $T$ is a $k$-scheme and $t \in T$ is a $k$-rational point. 
\begin{enumerate}
\item
A \textit{deformation} of $X$ over $T$ with reference point $t$ is a pair $(\mathcal{X}, i)$ of a scheme $\mathcal{X}$ that is flat and of finite type over $T$ and an isomorphism $i:X  \xrightarrow{\ \sim\ }  \mathcal{X} \times_T \Spec \kappa(t)$ of $k$-schemes. 
\item  Let $Z$ be a closed subscheme of $X$. 
A \textit{deformation} of the pair $(X, Z)$ over $T$ with reference point $t$ is a quadruple $(\mathcal{X}, i,\mathcal{Z}, j)$ where $(\mathcal{X}, i)$ is a deformation of $X$ over $T$ with reference point $t$, $\mathcal{Z}$ is a closed subscheme of $\mathcal{X}$ that is flat over $T$ and $j$ is an isomorphism $j:Z \xrightarrow{\ \sim\ } \mathcal{Z} \times_{\mathcal{X}} X$ of $k$-schemes. 
\end{enumerate}
\end{defn}

In the later sections, we will use the following setup to consider some problems on deformations of singularities. 

\begin{setting}\label{local setting}
Suppose that $k$ is an algebraically closed field of characteristic zero, $X$ is a reduced $(S_2)$ and $(G_1)$ scheme of finite type over $k$,  $T$ is an irreducible scheme over $k$ with generic point $\eta$ and $t \in T$ is a closed point. 
Let $(\mathcal{X},i)$ be a deformation of $X$ over $T$ with reference point $t$ such that $\mathcal{X}$ is a reduced $(S_2)$ and $(G_1)$ scheme. 
Let $\mathcal{D} \in \WDiv^*_\Q(\mathcal{X})$ be an effective $\Q$-Weil divisor on $\mathcal{X}$ whose support does not contain any generic points of the closed fiber $X$ nor any singular codimension one points of $X$. 
Let $\ba \subseteq \sO_{\mathcal{X}}$ be a coherent ideal sheaf such that $\ba \sO_X$ is nonzero and $\lambda > 0$ be a real number.
\end{setting}

\section{Deformations of valuatively klt singularities}
In this section, we prove the inversion of adjunction for valuatively klt singularities. 
As a corollary, we show that valuatively klt singularities are invariant under a deformation over a smooth base, which is a generalization of a result of Esnault-Viehweg \cite{EV} on deformations of klt singularities. 

Throughout this section, we say that $(R, \Delta, \ba^{\lambda})$ is a \textit{triple} of equal characteristic zero if $(R,\m)$ is an excellent normal local ring of equal characteristic zero with a  dualizing complex $\omega_R^{\bullet}$, $\Delta$ is an effective $\Q$-Weil divisor on $\Spec R$, $\ba$ is a nonzero ideal of $R$ and $\lambda > 0$ is a real number.

\begin{prop}\label{adjoint ideal twists}
Suppose that $(R, \Delta, \ba^{\lambda})$ is a triple of equal characteristic zero and $D$ is a reduced Weil divisor on $X: = \Spec R$ 
such that $\ba$ is trivial at any generic points of $D$.  
Let $A$ be an effective Weil divisor on $X$ linearly equivalent to $- K_X-D$ such that  $B:=A-\Delta$ is also effective and $A$ has no common components with $D$.  
Fix an integer $m \ge 1$ such that $m \Delta$ is an integral Weil divisor.
\begin{enumerate}[label=$(\arabic*)$]
\item $\adj_D(X, A+D, \ba^\lambda \sO_X(-mB)^{1-1/m})$ is contained in $\sO_X(-m B)$. 
\item The following conditions are equivalent to each other.
\begin{enumerate}[label=\textup{(\alph*)}]
\item $(X,\Delta+D, \ba^{\lambda})$ is $m$-weakly valuatively plt along $D$.
\item For every nonzero coherent ideal $\bb \subseteq \sO_X$ contained in $\sO_X(-mB)$ that is trivial at any generic points of $D$, we have 
\[
\bb \subseteq \adj_D(X, A+D, \ba^\lambda \bb^{1-1/m}) .
\]
\item For every nonzero principal ideal $(r) \subseteq \sO_X$ contained in $\sO_X(-mB)$ that is trivial at any generic points of $D$, we have 
\[
r \in \adj_D(X, A+D, \ba^\lambda (r)^{1-1/m}) .
\] 
\item For every anti-effective $\Q$-Weil divisor $\Gamma$ on $X$ such that $m(K_X+ \Delta + \Gamma +D)$ is Cartier and $\Gamma$ has no common components with $D$, the triple $(X, \Delta + \Gamma +D , \ba^{\lambda})$ is subplt along $D$, that is, $\sO_X \subseteq \adj_D(X, \Delta+\Gamma+D, \ba^{\lambda})$.
\item $\adj_D(X, A+D, \ba^\lambda \sO_X(-m B)^{1-1/m}) =\sO_X(-mB)$.
\end{enumerate}
\end{enumerate}
\end{prop}

\begin{proof}
(1) Let $U \subseteq X$ denote the locus where $mB$ is Cartier.
Since $\sO_X(-mB)$ is reflexive and $U$ is an open subset of $X$ whose complement has codimension at least two, it suffices to show that 
\[
\adj_D(X, A+D, \ba^{\lambda} \sO_X(-mB)^{1-1/m})|_U \subseteq \sO_U(-m B|_U).
\]
However, it follows from the fact that $\sO_U(-m B|_U)$ is invertible that 
\begin{align*}
\adj_D(X, A+D, \ba^{\lambda} \sO_X(-mB)^{1-1/m})|_U 
&\subseteq \adj_{D|_U}(U, A|_U +D|_U, \ba|_U^\lambda \sO_U(-mB|_U)^{1-1/m})\\
&= \adj_{D|_U} (U, A|_U +\frac{m-1}{m} (mB|_U) +D|_U, \ba|_U^\lambda)\\
&= \adj_{D|_U} (U, \Delta|_U + m B|_U +D|_U, \ba|_U^\lambda)\\
&= \adj_{D|_U} (U, \Delta|_U +D|_U, \ba|_U ^\lambda) \otimes_{\sO_U} \sO_U(-m B|_U)\\
& \subseteq \sO_U(-m B|_U). 
\end{align*}

(2) First we prove the implication $($a$) \Rightarrow ($b$)$.
Take a log resolution $f: Y \to X$ of $(X, \Delta+A+D,   \ba \bb \sO_X(-mB))$ separating the components of $D$, and write 
\[
\ba \sO_Y= \sO_Y(-F), \ \bb \sO_Y=\sO_Y(-G) \textup{ and } \sO_X(-mB) \sO_Y = \sO_Y(-H).
\]
We also set
\begin{align*}
(A+D)_Y &: = f^*(K_X+A+D)-K_Y, \\
(\Delta+D)_{m,Y}^+ &:=  - \frac{f^{\natural}(-m(K_X+\Delta+D))}{m} - K_Y.
\end{align*}
Since 
\begin{align*}
H-m(K_Y+(A+D)_Y) &= f^{\natural}(mB) - f^*(m(K_X+A+D)) \\
&= f^{\natural}(m(B - (K_X+A+D)))\\ 
&= f^{\natural}(-m(K_X+\Delta+D)),
\end{align*}
one has 
\[
(\Delta+D)_{m,Y}^+ = -\frac{1}{m}H + (A+D)_Y.
\]
Therefore, we obtain 
\begin{align*}
&\adj_D(X,A+D, \ba^{\lambda} \bb^{1-1/m})\\
=& f_* \sO_Y\left(-\left\lfloor (A+D)_Y - f^{-1}_*D + \lambda F + \left(1-\frac{1}{m}\right) G \right\rfloor\right) \\
=& f_* \sO_Y\left(- \left\lfloor  (\Delta+D)_{m,Y}^{+} + \frac{1}{m}H - f^{-1}_*D + \lambda F + \left(1-\frac{1}{m}\right) G \right\rfloor\right) \\
=& f_* \sO_Y\left(-G - \left\lfloor (\Delta+D)^+_{m,Y} - f^{-1}_*D + \lambda F -\frac{1}{m}(G-H) \right\rfloor \right).
\end{align*}
Combining this with the inequalities $H \le G$ and $\lfloor (\Delta+D)^+_{m,Y} - f^{-1}_*D + \lambda F \rfloor \le 0$ yields the inclusion 
\[
\adj_D(X, A+D, \ba^\lambda \bb^{1-1/m}) \supseteq f_* \sO_Y(-G) \supseteq \bb. 
\]

Next we prove the implication (e) $\Rightarrow$ (a). 
An argument similar to the above and (e) show that 
\begin{align*}
f_* \sO_Y(-H - \lfloor (\Delta+D)^+_{m,Y} - f^{-1}_*D + \lambda F \rfloor )
&=\adj_D(X,A +D, \ba^{\lambda} \sO_Y(-mB)^{1-1/m})\\
&=\sO_X(-mB)\\
&=f_*\sO_Y(-H). 
\end{align*}
Since $\sO_Y(-H)$ is globally generated with respect to $f$, we conclude that 
\[
\lfloor (\Delta+D)^+_{m,Y} - f^{-1}_*D  + \lambda F \rfloor \le 0,
\]
which proves (a) by Lemma \ref{single resolution}.

The implication (b) $\Rightarrow$ (c) is obvious.
For (c) $\Rightarrow$ (e), take a system of generators $r_1, \dots, r_n$ of the ideal $\sO_X(- m B) \subseteq R$.
Since $B$ has no common components with $D$, replacing the generators by their linear combinations, we may assume that the principal ideal $(r_i)$ is trivial at any generic points of $D$ for all $i$.
Then (e) follows from (1) and an application of (c) with $r=r_i$. 

For (c) $\Rightarrow$ (d), take an anti-effective $\Q$-Weil divisor $\Gamma$ such that $m(K_X+\Delta+\Gamma +D)$ is Cartier and $\Gamma$ has no common components with $D$.  
Since $R$ is local and $m(K_X+\Delta+D)$ is linearly equivalent to $-mB$, we may write
\[
-m B +m \Gamma  + \Div_X(r) = 0,
\]
where $r$ is an element of $\mathrm{Frac}(R)$.
Since $\Gamma$ is anti-effective and $B-\Gamma$ has no common components with $D$, the principal ideal $(r)$ is contained in $\sO_X(-mB)$ and is trivial at any generic points of $D$. 
Therefore, applying (c) to this principal ideal, we obtain 
\begin{align*}
r \in \adj_D(X, A+D, \ba^{\lambda} (r)^{1-1/m}) & = \adj_D(X, \Delta+ \Gamma + \Div_X(r) +D, \ba^{\lambda}) \\
&= r \cdot \adj_D(X, \Delta+ \Gamma +D, \ba^{\lambda}) ,
\end{align*}
which implies that 
\[
1 \in \adj_D(X, \Delta+ \Gamma +D, \ba^{\lambda}).
\]

For the converse implication (d) $\Rightarrow$ (c), just reverses the above argument. 
\end{proof}

The following theorem is the main result of this section, which shows the inversion of adjunction for valuatively klt singularities. 
\begin{thm}\label{vklt cut}
Suppose that $(R,\Delta, \ba^\lambda)$ is a triple of equal characteristic zero and $h$ is a nonzero element in $R$ such that $S:= R/(h)$ is normal. 
We assume in addition that $Z:=\Spec S$ is not contained in the support of $\Delta$ and $\ba$ is not contained in the ideal $(h)$.
\begin{enumerate}[label=$(\arabic*)$]
\item Let $m \ge 1$ be an integer such that $m \Delta$ is an integral Weil divisor.
If the triple $(Z, \Delta|_Z, (\ba S)^\lambda)$ is $m$-weakly valuatively klt, then $(X, \Delta +Z, \ba^{\lambda})$ is $m$-weakly valuatively plt along $Z$.
\item If $(Z, \Delta|_Z, (\ba S)^\lambda)$ is valuatively klt, then $(X, \Delta +Z, \ba^{\lambda})$ is valuatively plt along $Z$, and in particular, $(X, \Delta, \ba^{\lambda})$ is valuatively klt.  
\end{enumerate}
\end{thm}

\begin{proof}
(1) Since $X$ is affine and Gorenstein at the generic point of $Z$, we can take an effective Weil divisor $A$ on $X$ linearly equivalent to $-K_X$ such that $B:= A -\Delta$ is effective and $\Supp A$ does not contain $Z$.
Take an integer $m \ge 1$ such that $m (K_Z + \Delta|_Z)$ is Cartier.
Set
\begin{align*}
I &: = \adj_Z(X, A+Z, \ba^\lambda \bb^{1-1/m}) \subseteq R, \\ 
J &:= \mathcal{J}(Z, A|_Z, (\ba S)^\lambda (\bb S)^{1-1/m}) \subseteq S,
\end{align*}
where $\bb : = \sO_X(-mB) \subseteq R$.

Since $A|_Z$ is linearly equivalent to $-K_Z$, $B|_Z=A|_Z - \Delta|_Z$ and $\bb S \subseteq \sO_Z(-m B|_Z)$, 
we apply Proposition \ref{adjoint ideal twists} (2) (a)$\Rightarrow$(b) with $X=Z$ and $D=\emptyset$ to deduce that 
\[
\bb S \subseteq J = IS, 
\]
where the last equality is a consequence of the restriction theorem \cite[Theorem 1.5]{Ta}.\footnote{\cite[Theorem 1.5]{Ta} is formulated for varieties, but the same statement for excellent $\Q$-schemes is obtained by using \cite[Theorem A]{Mur} instead of the local vanishing theorem.}
It follows from a combination of the inclusion $\bb S \subseteq IS$ with Proposition \ref{adjoint ideal twists} (1) that 
\[
I \subseteq \bb \subseteq I + \bb \cap (h).
\]
By assumption, $\Div_X(h)=Z$ is a prime divisor on $X$, which is not an irreducible component of $B$. Thus, $\bb \cap (h)=h(\bb:_R(h))=h \bb \subseteq \m \bb$, so that $\bb=I+\m \bb$. 
By Nakayama's lemma, we have $I=\bb$, which completes the proof by using Proposition \ref{adjoint ideal twists} (2) (e)$\Rightarrow$(a). 

(2) Take an ideal $I \subseteq \sO_X$ such that $I \sO_Z$ is nonzero and the closed subset $V(I) \subseteq X$ contains the singular locus of $X$ and that of $Z$.
The assertion then follows from (1), Proposition \ref{weak sense vs strong sense} and Remark \ref{valuative implication}. 
\end{proof}

\begin{rem}
The no boundary case, that is, the case where $\Delta=0$ and $\ba=R$, of Theorem \ref{vklt cut} (2) was originally claimed in \cite[Theorem 3.8]{Chi}, but there is an error in the proof.
Our proof is completely different from the one given there. 
  \end{rem}

\begin{cor}\label{vklt cut2}
Suppose that $(R,\Delta, \ba^\lambda)$ is a triple of equal characteristic zero and $h_1, h_2,$ $\dots, h_r$ forms a regular sequence of $R$ such that $S:= R/(h_1, \dots, h_r)$ is normal. 
We assume in addition that $Z:=\Spec S$ is not contained in the support of $\Delta$ and $\ba$ is not contained in the ideal $(h_1, \dots, h_r)$.
If $(Z, \Delta|_Z, (\ba S)^\lambda)$ is valuatively klt, then so is $(X, \Delta, \ba^{\lambda})$.
\end{cor}

\begin{proof}
It follows from repeated applications of Theorem \ref{vklt cut} (2).
\end{proof}

\begin{cor}\label{vklt local}
With notation as in Setting \ref{local setting}, we assume that $X$ and $\mathcal{X}$ are normal integral schemes.
let $x \in X$ be a closed point and $\mathcal{Z} \subseteq X$ be an irreducible closed subscheme such that $(\mathcal{X}, i, \mathcal{Z}, j)$ is a deformation of the pair $(X, \{x\}_{\mathrm{red}})$ over $T$ with reference point $t$. 
Let $y$ be the generic point of $\mathcal{Z}$, which lies in the generic fiber $\mathcal{X}_\eta$. If $(X, \mathcal{D}|_X, (\ba \sO_X)^\lambda)$ is valuatively klt at $x$, then so is $(\mathcal{X}_\eta, \mathcal{D}_\eta, (\ba \sO_{\mathcal{X}_\eta})^\lambda)$ at $y$.
\end{cor}

\begin{proof}
Let $f : \widetilde{T} \to T_{\mathrm{red}}$ be a resolution of singularities of the reduced closed subscheme $T_{\mathrm{red}}$ of $T$.
Take a closed point $\widetilde{t} \in \widetilde{T}$ that maps to the point $t \in T$.
Since the closed fiber of $\mathcal{X} \times_T \widetilde{T}$ over $\widetilde{t}$ is isomorphic to $\mathcal{X}_t =X$ and the generic fiber of $\mathcal{X} \times_T \widetilde{T}$ is isomorphic to $\mathcal{X}_{\eta}$, after replacing $T$ by $S$, we may assume that $T$ is a regular integral scheme.
Then the closed fiber $X$ is locally a complete intersection in $\mathcal{X}$, and we see from Corollary \ref{vklt cut2} that $(\mathcal{X}, \mathcal{D}, \ba^\lambda)$ is valuatively klt at $x$. 
Since $y$ is a generalization of $x$, the triple $(\mathcal{X}, \mathcal{D}, \ba^\lambda)$ is valuatively klt at $y$ by Lemma \ref{singularity at a point}, which completes the proof. 
\end{proof}

\begin{cor}\label{vklt proper}
With notation as in Setting \ref{local setting}, we assume that $X$ and $\mathcal{X}$ are normal integral schemes.
We further assume that $\mathcal{X}$ is proper over $T$.
If $(X, \mathcal{D}|_X, (\ba \sO_X)^\lambda)$ is valuatively klt, then so is $(\mathcal{X}_\eta, \mathcal{D}_\eta, (\ba \sO_{\mathcal{X}_\eta})^\lambda)$.
\end{cor}

\begin{proof}
Since the structure map $\mathcal{X} \to T$ is a closed map, it follows from an argument similar to the proof of Corollary \ref{vklt local} that $(\mathcal{X}, \mathcal{D}, \ba^\lambda)$ is valuatively klt near $\mathcal{X}_{\eta}$.   
\end{proof}

Finally, we show that valuatively klt singularities are invariant under deformations.  
Corollary \ref{general klt fibers} (2) gives an alternative proof of a result of Esnault-Viehweg \cite{EV}. 

\begin{cor}\label{general klt fibers}
Let $T$ be an irreducible algebraic scheme over an algebraically closed field $k$ of characteristic zero and $(\mathcal{X}, \mathcal{D}, \ba^{\lambda}) \to T$ be a proper flat family of triples over $T$, where $\mathcal{D}$ is an effective $\Q$-Weil divisor on a normal variety $\mathcal{X}$ over $k$, $\ba \subseteq \sO_X$ is a nonzero coherent ideal sheaf and $\lambda > 0$ is a real number. 
\begin{enumerate}[label=\textup{(\arabic*)}]
\item
If some closed fiber $(\mathcal{X}_{t_0}, \mathcal{D}_{t_0}, (\ba \sO_{\mathcal{X}_{t_0}})^{\lambda})$ is valuatively klt, then so is a general closed fiber $(\mathcal{X}_{t}, \mathcal{D}_{t}, (\ba \sO_{\mathcal{X}_{t}})^{\lambda})$.
\item 
If some closed fiber $(\mathcal{X}_{t_0}, \mathcal{D}_{t_0}, (\ba \sO_{\mathcal{X}_{t_0}})^{\lambda})$ is two-dimensional klt, then so is a general closed fiber $(\mathcal{X}_{t}, \mathcal{D}_{t}, (\ba \sO_{\mathcal{X}_{t}})^{\lambda})$. 
\end{enumerate}
\end{cor}

\begin{proof}
(1) 
It follows from Corollary \ref{vklt proper} that the generic fiber $(\mathcal{X}_{\eta}, \mathcal{D}_{\eta}, (\ba \sO_{X_\eta})^{\lambda})$ is valuatively klt.
Then by Lemma \ref{single resolution},
a general closed fiber $(\mathcal{X}_{t}, \mathcal{D}_{t}, (\ba \sO_{\mathcal{X}_{t}})^{\lambda})$ is also valuatively klt. 

(2) 
Corollary \ref{vklt proper} and Lemma \ref{2-dim val lc} tell us that the generic fiber $(\mathcal{X}_{\eta}, \mathcal{D}_{\eta}, (\ba \sO_{X_\eta})^{\lambda})$ is klt. 
Then a general closed fiber $(\mathcal{X}_{t}, \mathcal{D}_{t}, (\ba \sO_{\mathcal{X}_{t}})^{\lambda})$ is also klt. 
\end{proof}

\section{Deformations of  slc singularities}

In this section, we study small deformations of slc singularities. 

\subsection{Variants of non-lc ideal sheaves}
Fujino's non-lc ideal sheaves are a generalization of multiplier ideal sheaves that defines non-lc locus (see \cite{Fuj} and \cite{FST}). 
We introduce two new variants of these ideal sheaves to generalize the inversion of adjunction for slc singularities. 
This subsection is devoted to their definitions and basic properties. 

Throughout this subsection, we assume that $\Gamma $ is an $\R$-Weil divisor, $W$ is a reduced Weil divisor and $D$ is a Weil divisor on a normal integral scheme $X$.
\begin{defn}
\ 
\begin{enumerate}[label=(\roman*)]
\item The $\Q$-Weil divisor $\Theta^W(\Gamma)$  on $X$ is defined as
\[
\Theta^W(\Gamma) : =  \Gamma -\sum_E E,
\]
where $E$ runs through all irreducible components of $W$ such that $\ord_E(\Gamma)$ is an integer.
\item The $\Q$-Weil divisor $\Theta^W_D(\Gamma)$ on $X$ is defined as
\[
\Theta^W_D(\Gamma) : =  \Gamma -\sum_E E,
\]
where $E$ runs through all irreducible components of $W$ such that $\ord_E(\Gamma) = \ord_E(D) +1$.
\end{enumerate}
\end{defn}

We collect some basic properties of $\Theta^W(\Gamma)$ and $\Theta^W_D(\Gamma)$ in the following lemma.  
\begin{lem}\label{Theta basic}
\ 
\begin{enumerate}[label=$(\arabic*)$]
\item $\Theta^W(\Gamma) \le \Theta^W_D(\Gamma)$.
\item For a Weil divisor $A$ on $X$, we have 
\[
\Theta^W(\Gamma+A) = \Theta^W(\Gamma) +A \; \; \textup{and} \; \; \Theta^W_{D+A}(\Gamma+A) = \Theta^W_D(\Gamma)+A.
\]
\item For a reduced Weil divisor $W'$ on $X$ such that $W \le W'$, we have 
\[
\Theta^W(\Gamma) \ge \Theta^{W'}(\Gamma)  \; \; \textup{and}  \; \; \Theta^W_D(\Gamma) \ge \Theta^{W'}_D(\Gamma).
\]
\item For an $\R$-Weil divisor $\Gamma'$ on $X$ such that $\Gamma \le \Gamma'$, we have 
\[
\lfloor \Theta^W(\Gamma) \rfloor \le \lfloor \Theta^W(\Gamma') \rfloor \; \; \textup{and} \; \; \lfloor \Theta^W_D(\Gamma) \rfloor \le \lfloor \Theta^W_D(\Gamma') \rfloor.
\]
\item For an open subscheme $U \subseteq X$, we have
\[
\Theta^W(\Gamma)|_U = \Theta^{W|_U}(\Gamma|_U) \; \; \textup{and} \; \; \Theta^W_D(\Gamma)|_U = \Theta^{W|_U}_{D|_U}(\Gamma|_U).
\]
\item For a reduced Weil divisor $W''$ having no common components with $W$, we have
\[
\Theta^{W + W''}(\Gamma) = \Theta^W(\Theta^{W''}(\Gamma)) \; \; \textup{and} \; \; \Theta^{W + W''}_D(\Gamma) = \Theta^W_D(\Theta^{W''}_D(\Gamma)).
\]
\end{enumerate}
\end{lem}

\begin{proof}
The proof is straightforward.
\end{proof}

\begin{lem}\label{Theta birat}
Suppose that $X$ is a regular integral excellent scheme with dualizing complex and the union of the supports of $\Gamma$, $D$ and $W$ is a simple normal crossing divisor, which is denoted by $B$. 
Let $f: Y \to X$ be a log resolution of $(X, B)$ with exceptional locus $\mathrm{Exc}(f)=\bigcup_i E_i$, and set $\Gamma_Y : = f^*(K_X+\Gamma)-K_Y$ and $W_Y : = f^{-1}_*W + \sum_i E_i$. 
Then 
\begin{align*}
f_*\sO_Y(- \lfloor \Theta^{W_Y}(\Gamma_Y) \rfloor)&=\sO_X(- \lfloor \Theta^{W}(\Gamma) \rfloor),\\
f_*\sO_Y(- \lfloor \Theta^{W_Y}_{f^* D}(\Gamma_Y) \rfloor)&=\sO_X(- \lfloor \Theta^{W}_D(\Gamma) \rfloor). 
\end{align*}
\end{lem}

\begin{proof}
The proof is similar to that of \cite[Lemma 2.7]{Fuj}.
\end{proof}

We are now ready to define our variants of Fujino's non-lc ideal sheaves. 
\begin{defn}\label{new ideal}
Suppose that $X$ is a normal variety over a field $k$ of characteristic zero and $D$ is a Cartier divisor on $X$. 
Let $\Delta$ be an effective $\Q$-Weil divisor on $X$ such that $K_X+\Delta$ is $\Q$-Cartier, $\ba \subseteq \sO_X$ be a nonzero coherent ideal sheaf and $\lambda > 0$ be a real number.
Let $B$ 
be the union of the supports of $\Delta$, $W$ and $D$, and take a log resolution $f: Y \to X$ of $(X,B, \ba)$ with $\ba\sO_Y=\sO_Y(-F)$ and $\mathrm{Exc}(f)=\bigcup_i E_i$.  
The fractional ideal sheaves $\mathcal{J}^W(X, \Delta, \ba^\lambda)$ and $\mathcal{J}^W_D(X, \Delta, \ba^\lambda)$ are then defined as
\begin{align*}
\mathcal{J}^W(X, \Delta, \ba^\lambda) &: = f_* \sO_Y(-\lfloor \Theta^{W_Y}(\Delta_Y + \lambda F) \rfloor), \\
\mathcal{J}^W_D(X, \Delta, \ba^\lambda) &: = f_* \sO_Y(-\lfloor \Theta^{W_Y}_{f^*D}(\Delta_Y + \lambda F) \rfloor) ,
\end{align*}
where $W_Y: = f^{-1}_*W+ \sum_i E_i$ and $\Delta_Y : =f^*(K_X+\Delta)- K_Y$. 
This definition is independent of the choice of the log resolution $f$ by Lemma \ref{Theta birat}.

When $W$ is the union of the support of $D+\Delta$ and all the codimension one irreducible components of the closed subscheme of $X$ defined by $\ba$, the fractional ideal sheaf $\mathcal{J}^{W}_D(X, \Delta, \ba^\lambda)$ is denoted simply by $\mathcal{J}_D(X, \Delta, \ba^\lambda)$. 
\end{defn}

\begin{rem}
Definition \ref{new ideal} makes sense even if $X$ is disconnected. We also remark that given a nonzero coherent ideal sheaf $\bb \subseteq \sO_X$ and a real number $\lambda'>0$, the fractional ideal sheaves $\mathcal{J}^W(X, \Delta, \ba^\lambda \bb^{\lambda'})$ and  $\mathcal{J}^W_D(X, \Delta, \ba^\lambda \bb^{\lambda'})$ are defined similarly. 
\end{rem}

\begin{rem}
Let $(X,\Delta, \ba^\lambda)$ be as in Definition \ref{new ideal}, and assume in addition that $W$ is the union of the support of $\Delta$ and all the codimension one irreducible components of the closed subscheme of $X$ defined by $\ba$. 
\begin{enumerate}[label=(\roman*)]
\item  $\mathcal{J}^W(X, \Delta, \ba^\lambda)$ coincides with the maximal non-lc ideal sheaf $\mathcal{J}'(X, \Delta, \ba^\lambda)$ defined in \cite{FST}.
\item  $\mathcal{J}^W_0(X, \Delta, \ba^\lambda)$ coincides with the non-lc ideal sheaf $\mathcal{J}_{\mathrm{NLC}}(X,\Delta, \ba^\lambda)$  defined in \cite{Fuj}.
\end{enumerate}
\end{rem}

The following two lemmas state basic properties of $\mathcal{J}^W(X, \Delta, \ba^\lambda)$ and $\mathcal{J}^W_D(X, \Delta, \ba^\lambda)$ that we will use later. 
\begin{lem}\label{J basic}
Let $(X,\Delta, \ba^\lambda)$, $W$ and $D$ be as in Definition \ref{new ideal}.
\begin{enumerate}[label=$(\arabic*)$]
\item $\mathcal{J}^W_D(X, \Delta, \ba^\lambda) \subseteq \mathcal{J}^W(X, \Delta, \ba^\lambda)$.
\item  For a Cartier divisor $A$ on $X$, we have
\begin{align*}
\mathcal{J}^W(X, \Delta+A, \ba^\lambda) &= \mathcal{J}^W(X, \Delta, \ba^\lambda)\otimes_X \sO_X(-A),\\ 
\mathcal{J}^W_{D+A}(X, \Delta+A, \ba^\lambda) &= \mathcal{J}^W_D(X, \Delta, \ba^\lambda) \otimes_X \sO_X(-A).
\end{align*}
\item For a reduced Weil divisor $W'$ on $X$ such that $W \le W'$, we have
\begin{align*}
\mathcal{J}^W(X, \Delta, \ba^\lambda) &\subseteq \mathcal{J}^{W'}(X, \Delta, \ba^\lambda),\\ 
\mathcal{J}^W_D(X, \Delta, \ba^\lambda) &\subseteq \mathcal{J}^{W'}_D(X, \Delta, \ba^\lambda).
\end{align*}
\item Let $\Delta'$ be an effective $\Q$-Weil divisor on $X$ such that $K_X+\Delta'$ is $\Q$-Cartier and $\Delta \le \Delta'$, $\ba'$ be a nonzero coherent ideal sheaf such that $\ba \supseteq \ba'$ and $\lambda'$ be a real number such that $\lambda \le \lambda'$.
Then 
\begin{align*}
\mathcal{J}^W(X, \Delta, \ba^\lambda) &\supseteq \mathcal{J}^W(X', \Delta', {\ba'}^{\lambda'}),\\ 
\mathcal{J}^W_D(X, \Delta, \ba^\lambda) &\supseteq \mathcal{J}^W_D(X', \Delta', {\ba'}^{\lambda'}).
\end{align*}
\item For an open subscheme $U \subseteq X$, we have
\begin{align*}
\mathcal{J}^W(X, \Delta, \ba^\lambda)|_U &= \mathcal{J}^{W|_U}(U, \Delta|_U, \ba|_U^\lambda), \\ 
\mathcal{J}^W_D(X, \Delta, \ba^\lambda)|_U &= \mathcal{J}^{W|_U}_{D|_U}(U, \Delta|_U, \ba|_U^\lambda).
\end{align*}
\end{enumerate}
\end{lem}

\begin{proof}
All the assertions immediately follow from Lemma \ref{Theta basic}.
\end{proof}

\begin{lem}\label{J vs lc}
Let $(X,\Delta, \ba^\lambda)$, $W$ and $D$ be as in Definition \ref{new ideal}.
Let $G$ 
be the cycle of codimension one in $X$ associated to the closed subscheme defined by $\ba$, that is, 
\[
G= \sum_{E} \ord_E(\ba) E,
\]
where $E$ runs through all prime divisors on $X$.
\begin{enumerate}[label=$(\arabic*)$]
\item We have inclusions 
\begin{align*}
\mathcal{J}^W(X, \Delta, \ba^\lambda) &\subseteq \sO_X(- \lfloor \Theta^W(\Delta + \lambda G) \rfloor ), \\
\mathcal{J}^W_D(X, \Delta, \ba^\lambda) &\subseteq \sO_X(- \lfloor \Theta^W_D(\Delta + \lambda G) \rfloor ).
\end{align*}
\item Assume that $W$ is contained in the support of $\Delta+G$.
Then following conditions are equivalent to each other:
\begin{enumerate}[label=\textup{(\alph*)}]
\item $\mathcal{J}^W(X,\Delta, \ba^\lambda) =\sO_X$,
\item $\mathcal{J}^W_0(X, \Delta, \ba^\lambda)=\sO_X$,
\item $(X,\Delta, \ba^\lambda)$ is lc and $\ord_E(\Delta+ \lambda G) < 1 $ for every prime divisor $E$ on $X$ that is not a component of $W$. 
\end{enumerate}
\end{enumerate}
\end{lem}

\begin{proof}
We use the notation established in Theorem \ref{new ideal}. 

(1) Since $f_* W_Y = W$, $f_* \Delta_Y = \Delta$ and $f_* F =G$, one has 
\[
f_* (\Theta^{W_Y}(\Delta_Y + \lambda F )) = \Theta^W(\Delta + G),
\] 
which implies the first inclusion $\mathcal{J}^W(X, \Delta, \ba^\lambda) \subseteq \sO_X(- \lfloor \Theta^W(\Delta + \lambda G) \rfloor )$. 
The second inclusion is shown similarly. 

(2) First note that the fractional ideals $\mathcal{J}^W(X, \Delta, \ba^\lambda)$ and $\mathcal{J}^W_D(X, \Delta, \ba^\lambda)$ are ideals in $\sO_X$ by (1) and the assumption that $W$ is contained in the support of $\Delta+G$.
Therefore, (a) (resp. (b)) holds if and only if $\lfloor \Theta^{W_Y}(\Delta_Y + \lambda F ) \rfloor \le 0$ (resp. $\lfloor \Theta^{W_Y}_0(\Delta_Y + \lambda F) \rfloor \le 0$).
It is easy to see that these inequalities are equivalent to (c). 
\end{proof}

\subsection{An extension of Inversion of adjunction for slc singularities}
In this subsection, we prove an extension of the inversion of adjunction for slc singularities to the non-$\Q$-Gorenstein setting, using our variants of Fujino's non-lc ideal sheaves. 
As a corollary, we show that slc singularities deform to lc singularities if the total space is normal and the nearby fibers are $\Q$-Gorenstein, which is a generalization of a result of S.~Ishii \cite{ish}. 

First we show an analog of \cite[Proposition 4.1]{ST} for our variants of non-lc ideal sheaves. 
\begin{prop}\label{non-lc ideal twists}
Suppose that $(R, \m)$ is a normal local ring essentially of finite type over a field of characteristic zero, $\Delta$ is an effective $\Q$-Weil divisor on $X=\Spec R$, $\ba \subseteq R$ is a nonzero ideal and $\lambda > 0$ is a real number. 
Let $W$ be the reduced Weil divisor on $X$ whose support coincides with the union of the supports of $\Delta$ and the cycle of codimension one in $X$ associated to the closed subscheme defined by $\ba$. 
Let $A$ be an effective Weil divisor on $X$ linearly equivalent to $-K_X$ such that $B:=A-\Delta$ is also effective. 
Fix an integer $m \ge 1$ such that $m \Delta$ is an integral Weil divisor, and let $\bb \subseteq R$ is a nonzero ideal contained in $\sO_X(-mB)$.  
\begin{enumerate}[label=$(\arabic*)$]
\item $\mathcal{J}^W(X, A, \ba^\lambda \bb^{1-1/m})$ is contained in $\sO_X(-m B)$. 
\item If $\mathcal{J}^W(X, A, \ba^\lambda \sO_X(-m B)^{1-1/m}) =\sO_X(-mB)$, then $(X,\Delta, \ba^\lambda)$ is $m$-weakly valuatively lc.
\item Assume that $m (K_X+\Delta)$ is Cartier. 
If $(X, \Delta, \ba^\lambda)$ is lc, then 
$\bb$ is contained in $\mathcal{J}^W_{mB}(X, A, \ba^\lambda \bb^{1-1/m})$. 
\end{enumerate}
\end{prop}

\begin{proof}
(1) The assertion follows from arguments similar to the proof of \cite[Proposition 4.1]{ST} (1) by replacing \cite[Proposition 9.2.31]{Laz} with Lemma \ref{J basic} (ii). 

(2) Assume to the contrary that there exists a log resolution $f: Y \to X$ of $(X, A+W+\Delta, \ba,\sO_X(-mB))$ and a prime divisor $E$ on $Y$ such that $a_{m,E}^+(X, \Delta, \ba^\lambda)<-1$.
We write $\ba \sO_Y= \sO_Y(-F_a)$ and $\sO_X(-mB) \sO_Y = \sO_Y(-F_b)$. 
Since $K_X+A$ is Cartier, we have
\begin{align*}
\sO_X(m(K_X+\Delta)) \sO_Y &= \sO_X(m(K_X+A)-m B) \sO_Y\\
&=\sO_Y(m(K_Y+A_Y) -F_b), 
\end{align*}
where $A_Y : = f^*(K_X+A)- K_Y$.
Then 
\begin{align*}
a_{m, E}^+ (X, \Delta, \ba^{\lambda}) &= \ord_E ( K_Y + \frac{1}{m}(-m(K_Y+A_Y)+ F_b)  -\lambda F_a)\\
&= \ord_E (  -A_Y  - \lambda F_a + \frac{1}{m}F_b)\\
&= \ord_E(- A_Y  - \lambda F_a - \frac{m-1}{m} F_b) + \ord_E(F_b).
\end{align*}
By the assumption that $a_{m,E}^+(X, \Delta, \ba^{\lambda}) <-1$, one has
\[
\ord_E(\lfloor \Theta^{W_Y}(A_Y+ \lambda F + \frac{m-1}{m} F_b) \rfloor) > \ord_E(F_b),
\]
where $W_Y$ is the reduced divisor on $Y$ whose support is the union of the strict transform $f^{-1}_*W$ of $W$ and the exceptional locus $\mathrm{Exc}(f)$ of $f$. 
Therefore, 
\begin{align*}
\mathcal{J}^{W}(X, A, \ba^{\lambda}  \sO_X(-m B)^{1-1/m}) &= f_* \sO_Y(- \lfloor \Theta^{W_Y}(A_Y+ \lambda F_a + \frac{m-1}{m} F_b) \rfloor)\\
&\subsetneq f_* \sO_Y(-F_b)\\ 
&= \sO_X(-m B), 
\end{align*}
where the second strict containment follows from the fact that $\sO_Y(-F_b)$ is $f$-free. 
This is a contradiction. 

(3) First note that $mB$ is Cartier by assumption, and we set $\q : = \bb \otimes \sO_X(m B)$.
It then follows from Lemma \ref{J basic} (2) that the inclusion in (3) is equivalent to the inclusion 
$\q \subseteq \mathcal{J}^W_0 (X,  \Delta, \ba^{\lambda} \q^{1-1/m})$. 
Take a log resolution $f: Y \to X$ of $(X, W, \ba \q)$ with $\ba \sO_Y=\sO_Y(-F_a)$ and $\q \sO_Y=\sO_Y(-F_q)$. 
Since $(X,  \Delta, \ba^{\lambda})$ is lc, all the coefficients of $ \Delta_Y+ \lambda F_a$ is less than or equal to one, where $\Delta_Y:= f^*(K_X+\Delta)-K_Y$.
Noting that $F_q$ is an effective integral divisor on $Y$, we have 
\[
\lfloor \Theta^{W_Y}_0( \Delta_Y+ \lambda F_a + \frac{m-1}{m} F_q) \rfloor \le F_q, 
\]
where $W_Y$ is the reduced divisor on $Y$ whose support is the union of the strict transform $f^{-1}_*W$ of $W$ and the exceptional locus $\mathrm{Exc}(f)$ of $f$. 
Therefore, 
\[
\q \subseteq f_* \sO_Y(-F_q) \subseteq \mathcal{J}^W_0 (X, \Delta, \ba^{\lambda} \q^{1-1/m}).
\]
\end{proof}

We have the following restriction theorem for our variants of non-lc ideal sheaves. 

\begin{thm}\label{non-lc restriction}
Let $(R, \m)$ be a normal local ring essentially of finite type over an algebraically closed field of characteristic zero and $A$ be an effective $\Q$-Weil divisor on $X:=\Spec R$ such that $K_X+A$ is $\Q$-Cartier. 
Suppose that $h$ is a nonzero element in $R$ such that $S : = R/(h)$ is reduced and any irreducible component of $Z:=\Spec S$ is not contained in the support of $A$.
Let $\lambda, \lambda' >0$ be real numbers, $\ba, \bb \subseteq R$ be ideals that are trivial at any generic point of $Z$ and $W$ be a reduced Weil divisor on $X$ having no common components with $Z$.   
We assume that there exist an ideal $J \subseteq S$, an effective Cartier divisor $D$ on $Z^n$ and an open subset $U \subseteq Z$ satisfying the following three conditions:
\begin{enumerate}[label=\textup{(\roman*)}]
\item $J S^n \subseteq \sO_{Z^n}(-D) \cap \mathcal{J}_D(Z^n, \mathrm{Diff}_{Z^n}(A), (\ba S^n)^\lambda (\bb S^n)^{\lambda'})$,
\item $(J S^n)_x \ne \sO_{Z^n}(-D)_x$ for any point $x \in Z^n$ whose image in $Z$ is not contained in $U$,
\item $J|_U \subseteq \mathcal{J}^{W+Z}(X,A + Z, \ba^\lambda \bb^{\lambda'}) S|_U$. 
\end{enumerate}
Then we have $J \subseteq \mathcal{J}^{W+Z}(X,A +Z, \ba^\lambda \bb^{\lambda'}) S$. 
\end{thm}

\begin{proof}
Let $V$ be the complement of $U$ in $Z$. 
Take a log resolution $f : Y \to X$ of $(X, \ba, \bb)$ with $\ba \sO_Y= \sO_Y(-F_a)$ and $\bb \sO_Y= \sO_Y(-F_b)$ such that $f^{-1}(V)$ is a closed subset of pure codimension one in $Y$ and that the union of $f^{-1}(\Supp A)$, $f^{-1}(Z)$, $f^{-1}(W)$, $f^{-1}(V)$, the support of the divisor $F_a+F_b$ and the exceptional locus $\mathrm{Exc}(f)$ of $f$ is a simple normal crossing divisor on $Y$. 
Let $g$ and $h$ denote the induced morphisms $g: \widetilde{Z} \to Z$ and $h: \widetilde{Z} \to Z^n$, respectively, where $\widetilde{Z}$ is the strict transform of $Z$ on $Y$. 
Let $W_{Z^n}$ be the union of the support of $D+\mathrm{Diff}_{Z^n}(A)$ and all the codimension one irreducible components of the closed subscheme of $Z^n$ defined by $\ba \bb S^n$. 
After replacing $f$ by its blowing up, we may assume that $h^{-1}(W_{Z^n}) \subseteq \Exc(f)$. 
Then $h: \widetilde{Z} \to Z^n$ is a log resolution of $(Z^n, \mathrm{Diff}_{Z^n}(A)+W_{Z^n}, 
(\ba S^n) (\bb S^n))$ with $(\ba S^n) \sO_{\widetilde{Z}}= \sO_{\widetilde{Z}}(-F_a|_{\widetilde{Z}})$ and $(\bb S^n) \sO_{\widetilde{Z}}= \sO_{\widetilde{Z}}(-F_b|_{\widetilde{Z}})$. 

Let $W_Y$ (resp.~$W_{\widetilde{Z}}$) be the reduced divisor on $Y$ (resp.~$\widetilde{Z}$) whose support is the union of the strict transform $f^{-1}_* W$ (resp.~$h^{-1}_* W_{Z^n}$) and the exceptional locus of $f$ (resp.~$g$).
Since $h^{-1}(W_{Z^n})$ is contained in the exceptional locus of $f$, we have $W_Y|_{\widetilde{Z}} \ge W_{\widetilde{Z}}$.
We decompose $W_Y=W_Y^1 + W_Y^2$ as follows:
\begin{enumerate}
\item[$\textup{(a)}$] $f(W_Y^1) \subset V$, 
\item[$\textup{(b)}$] no irreducible components of $W_Y^2$ are mapped into $V$ by $f$.
\end{enumerate}

Set $\Gamma : = f^*(K_X+A+Z)- K_Y- \widetilde{Z}$. 
Since $\Gamma|_{\widetilde{Z}} = h^*( K_{Z^n}+ \mathrm{Diff}_{Z^n}(A)) - K_{\widetilde{Z}}$ (see \cite[Paragraph 4.7]{Kol}), 
\begin{align*}
\mathcal{J}_D(Z^n, \mathrm{Diff}_{Z^n}(A), (\ba S^n)^\lambda (\bb S^n)^{\lambda'}) 
&=\mathcal{J}^{W_{Z^n}}_D(Z^n, \mathrm{Diff}_{Z^n}(A), (\ba S^n)^\lambda (\bb S^n)^{\lambda'})\\
&= h_* \sO_{\widetilde{Z}} ( - \lfloor \Theta^{W_{\widetilde{Z}}}_{h^*D}(\Gamma|_{\widetilde{Z}} + \lambda F_a|_{\widetilde{Z}} + \lambda' F_b|_{\widetilde{Z}}) \rfloor)\\
& \subseteq h_* \sO_{\widetilde{Z}} ( - \lfloor \Theta^{W_Y|_{\widetilde{Z}}}_{h^*D}(\Gamma|_{\widetilde{Z}} + \lambda F_a|_{\widetilde{Z}} + \lambda' F_b|_{\widetilde{Z}}) \rfloor)\\
& \subseteq h_* \sO_{\widetilde{Z}} ( - \lfloor \Theta^{W_Y^1|_{\widetilde{Z}}}_{h^*D}(\Theta^{W_Y^2|_{\widetilde{Z}}}(\Gamma|_{\widetilde{Z}} + \lambda F_a|_{\widetilde{Z}} + \lambda' F_b|_{\widetilde{Z}})) \rfloor),
\end{align*}
where the third containment follows from Lemma \ref{Theta basic} (3) and the last containment does from Lemma \ref{Theta basic} (1), (4) and (6).
Setting $\Lambda : = \Theta^{W_Y^2}(\Gamma + \lambda F_a + \lambda' F_b)$ and noting that the union of the supports of $\widetilde{Z}$, $W_Y^2$ and $\Gamma+\lambda F_a + \lambda' F_b$ is a simple normal crossing divisor on $Y$, 
one has 
\[
\Lambda|_{\widetilde{Z}} = \Theta^{W_Y^2|_{\widetilde{Z}}}(\Gamma|_{\widetilde{Z}} + \lambda F_a|_{\widetilde{Z}} + \lambda' F_b|_{\widetilde{Z}}),
\]
and therefore,
\[
\mathcal{J}_D(Z^n, \mathrm{Diff}_{Z^n}(A), (\ba S^n)^\lambda (\bb S^n)^{\lambda'}) \subseteq h_* \sO_{\widetilde{Z}} ( - \lfloor \Theta^{W_Y^1|_{\widetilde{Z}}}_{h^*D}(\Lambda|_{\widetilde{Z}}) \rfloor).
\]
\begin{claim}
$J \subseteq g_*\sO_{\widetilde{Z}} (- \lfloor \Lambda|_{\widetilde{Z}} \rfloor )$. 
\end{claim}
\begin{proof}[Proof of Claim]
It is enough to show that  $J S^n \subseteq h_* \sO_{\widetilde{Z}} (- \lfloor \Lambda|_{\widetilde{Z}} \rfloor )$. 
Take a connected component $C$ of $Z^n$, and let $\widetilde{C}$ denote the corresponding component of $\widetilde{Z}$.
By the assumption (i), for any nonzero element $r \in H^0(C, J S^n|_C)$, 
\[
\Div_{\widetilde{C}}(r) \ge \lfloor \Theta^{W_Y^1|_{\widetilde{Z}}}_{h^*D}(\Lambda|_{\widetilde{Z}}) \rfloor|_{\widetilde{C}} \; \; \textup{and} \; \; \Div_{\widetilde{C}}(r)  \ge  h^* D|_{\widetilde{C}}.
\]
Fix any prime divisor $E$ on $\widetilde{C}$.
If $\ord_E( \Lambda|_{\widetilde{Z}}) = \ord_E(h^*D) + 1$ and $E$ is contained in $W_Y^1|_{\widetilde{Z}}$, 
then  $ g( E) \subseteq V$ by the definition of $W_Y^1$, and it therefore follows from the assumption (ii) that 
\[
\ord_E(r) \ge \ord_E( h^*D) + 1 = \ord_E( \Lambda|_{\widetilde{Z}}).
\]
If $\ord_E( \Lambda|_{\widetilde{Z}}) \neq \ord_E(h^*D) + 1$ or $E \not\subseteq W_Y^1|_{\widetilde{Z}}$, then 
\[
\ord_E( \Theta^{W_Y^1|_{\widetilde{Z}}}_{g^*D}(\Lambda|_{\widetilde{Z}}))=\ord_E( \Lambda|_{\widetilde{Z}}).
\]
Thus, we obtain the inequality $\Div_{\widetilde{C}}(r) \ge \lfloor \Lambda|_{\widetilde{Z}} \rfloor|_{\widetilde{C}}$, which implies $J S^n \subseteq h_* \sO_{\widetilde{Z}} (- \lfloor \Lambda|_{\widetilde{Z}} \rfloor )$. 
\end{proof}

By the above claim, we have the following commutative diagram:
\[
\xymatrix{
 f_*\sO_Y(- \lfloor \Theta^{W_Y^1}(\Lambda) \rfloor)   \ar^-{\alpha}[r] & g_* \sO_{\tilde{Z}}(- \lfloor (\Theta^{W_Y^1}(\Lambda)|_{\tilde{Z}} \rfloor)  \\
f_* \sO_Y(- \lfloor \Lambda \rfloor) \ar@{^{(}->}^-{}[u] \ar^-{\beta}[r] & g_* \sO_{\tilde{Z}}(- \lfloor \Lambda|_{\tilde{Z}} \rfloor) \ar@{^{(}->}^-{}[u] \\
& J. \ar@{^{(}->}^-{}[u].
}\]
Noting that $\mathcal{J}^{W+Z}(X,A+ Z, \ba^\lambda \bb^{\lambda'}) = f_* \sO_Y(- \lfloor \Theta^{W_Y^1}(\Lambda) \rfloor)$ by Lemma \ref{Theta basic} (6), we have
\[
\mathrm{Im}\; \beta \subseteq \mathrm{Im}\; \alpha = \mathcal{J}^{W+Z}(X,A+ Z, \ba^\lambda \bb^{\lambda'}) S.
\]
Take any element $r \in J$. 
In order to prove the assertion of this theorem, it suffices to prove that 
the morphism $\delta: J \hookrightarrow g_* \sO_Y(- \lfloor \Lambda|_{\widetilde{Z}} \rfloor) \to \mathrm{Coker}\; \beta$ sends $r$ to zero. 
Since $(f_* W_Y^1)|_U=0$, one has an inclusion $J|_U \subseteq \mathcal{J}^{W+Z}(X,A + Z, \ba^\lambda \bb^{\lambda'}) S|_U=\mathrm{Im}\; \beta|_U$ by the assumption (iii), which implies that 
the support of $\delta(r) \in \mathrm{Coker}\; \beta$ is contained in $V$.

On the other hand, by pushing forward the short exact sequence
\[
0 \to \sO_Y(- \lfloor \Lambda \rfloor -\widetilde{Z}) \to \sO_Y(- \lfloor \Lambda \rfloor) \to \sO_{\widetilde{Z}}(- \lfloor \Lambda|_{\widetilde{Z}} \rfloor) \to 0,
\]
we obtain an inclusion $\mathrm{Coker}\; \beta \subseteq R^1f_*\sO_Y(- \lfloor \Lambda \rfloor -\widetilde{Z})$.
Let $H$ 
be the $f$-semiample $\R$-divisor $- (K_Y +\Gamma + \widetilde{Z} + \lambda F_a + \lambda' F_b) $ on $Y$, $\Delta$ 
be the fractional part of the $\R$-divisor $\Gamma + \lambda F_a + \lambda' F_b$ and $B$ be the reduced divisor on $Y$ whose support is the union of all prime divisors $E$ such that $E \subseteq W_Y^2$ and $E \not\subseteq \Supp \Delta$.
Since $B+\Delta$ has simple normal crossing support and 
\[
-\lfloor \Lambda \rfloor -\widetilde{Z} = (K_Y+ B+ \Delta)+H,
\]
it follows from \cite[Theorem 3.2]{Amb} (see also \cite[Theorem 1.1]{Fuj0}) that if 
\[
\delta(r) \in  R^1f_* \sO_Y(-\lfloor \Lambda \rfloor -\widetilde{Z}) 
\] is an nonzero element, then the support of $\delta(r)$ contains $f(T)$, where $T$ is a stratum of the simple normal crossing pair $(Y, B)$. 
Taking into account that $B$ and $f^{-1}(V)$ have no common components and $B+f^{-1}(V)$ has simple normal crossings, we have $f(T) \not\subseteq V$, which contradicts the fact that  the support of $\delta(r)$ is contained in $V$.
Therefore, we conclude that $\delta(r)=0$ as desired.
\end{proof}

\begin{setting}\label{setting for inversion}
Let $(R,\m)$ be an equidimensional local ring essentially of finite type over an algebraically closed field of characteristic zero, $h \in R$ be a nonzero divisor, $\lambda>0$ be a real number and $\ba \subseteq R$ be an ideal with the following properties:
\begin{enumerate}
\item[(1)] $S : = R/(h)$ is reduced and satisfies $(S_2)$ and $(G_1)$. Therefore, so is $R$. 
\item[(2)] $\ba$ is nonzero at any generic point of $X: = \Spec R$ and is trivial at any generic point of $Z: =\Spec S$. 
\item[(3)] Any generic point of $Z$ is a regular point of $X$.
\end{enumerate}
Moreover, let $\Delta$ be an effective $\Q$-Weil divisor on $X$ contained in $\WDiv^*_\Q(X)$, and let $K_X$ and $K_Z$ be canonical divisors contained in $\WDiv^*(X)$ and $\WDiv^*(Z)$, respectively, which exist by Proposition \ref{canonical divisor exists} and Example \ref{biequidim vs equidim}. 
We further assume that  
\begin{enumerate}
\item[(4)] neither any generic points of $Z$ nor any codimension one singular points of $Z$ are contained in the support of $\Delta$, and
\item[(5)] $K_Z + \Delta|_Z$ is $\Q$-Cartier, where $\Delta|_Z$ denotes the $\Q$-Weil divisor $\mathrm{Diff}_Z(\Delta)$ (see Lemma \ref{restriction of divisor}).
\end{enumerate}
\end{setting}

The main result of this section is an extension of the inversion of adjunction for slc singularities to the non-$\Q$-Gorenstein setting, which is stated as follows.  
\begin{thm}\label{slc cut}
In setting \ref{setting for inversion}, if $(Z, \Delta|_Z, (\ba S)^{\lambda})$ is lc, then $(X,\Delta +Z, \ba^{\lambda})$ is valuatively lc.  
If we further assume the condition 
\begin{enumerate}
\item[$(6)$] the pull back $Z' : = Z \times_X X^n$ of $Z$ to the normalization $X^n = \Spec R^n$ of $X$ satisfies $(S_2)$, 
\end{enumerate}
then the slc case also holds, that is, if $(Z, \Delta|_Z, (\ba S)^{\lambda})$ is slc, then $(X,\Delta +Z, \ba^{\lambda})$ is valuatively slc. 
\end{thm}

\begin{proof}
We only consider the slc case, as the lc case follows essentially the same arguments. 

First note that since the morphism $Z' \to Z$ is finite and birational, $Z'$ is reduced and the normalization $Z^n=\Spec S^n$ of $Z$ is isomorphic to that of $Z'$.
We consider the following diagram:
\[
\xymatrix{
 X  & X^n \ar_-{\nu}[l] & \\
Z \ar@{^{(}->}^-{}[u]  & Z' \ar@{^{(}->}^-{}[u] \ar^-{\mu}[l] & Z^n=(Z')^n. \ar^-{\pi}[l] \ar@/^18pt/[ll]^{\rho} \ar_-{f}[ul]
}\]

By prime avoidance, we can take an effective Weil divisor $A \in \WDiv^*(X)$, linearly equivalent to $-K_X$, whose support contains neither any generic points of $Z$ nor any codimension one singular points of $Z$.
We may also assume that $B:= A -\Delta$ is effective.
Fix an integer $m \ge 1$ such that $m \Delta \in \WDiv^*_{\Q}(X)$ is an integral Weil divisor and $m(K_Z+\Delta|_Z) \in \WDiv^*_{\Q}(Z)$ is a Cartier divisor.
It then follows from Lemma \ref{restriction of divisor} (2) that $m B|_Z = m A|_Z - m \Delta|_Z \sim -m(K_Z +\Delta|_Z)$ is Cartier. 

We define the $\Q$-Weil divisors $\Delta_{X^n}$, $A_{X^n}$ and $B_{X^n}$ on $X^n$ as
\[
\Delta_{X^n} : = \nu^* \Delta + C_X, \ \ A_{X^n} : = \nu^* A + C_X \textup{ and } B_{X^n} : = A_{X^n}- \Delta_{X^n},
\] 
where $C_X$ is the conductor divisor of $\nu$ on $X^n$. 
Let $W$ be the reduced Weil divisor on $X^n$ whose support coincides with the union of the support of $\Delta_{X^n}$ and all the codimension one irreducible components of the closed subscheme of $X^n$ defined by $\ba R^n$. 
Since $A_{X^n} \sim -K_{X^n}$, it suffices to show by Proposition \ref{non-lc ideal twists} (2) that 
\begin{align*}
\bb &= \mathcal{J}^{W+Z'}(X^n, A_{X^n}, \sO_{X^n}(-Z')^1 (\ba R^n)^\lambda \bb^{1-1/m}) \\
&= \mathcal{J}^{W+Z'}(X^n, A_{X^n} +Z', (\ba R^n)^\lambda \bb^{1-1/m}) \tag{$\star$}
\end{align*}
where $\bb : = \sO_{X^n}(-mB_{X^n})$.

\begin{cln}
Let $S' = R^n/(h)$ be the structure ring of $Z'$ and $L \subseteq S'$ denote the principal ideal $\sO_{Z'}(-\mu^* (mB|_Z))$.
Then $\bb S' \subseteq L$. 
\end{cln}

\begin{proof}[Proof of Claim $1$]
Noting that $B_{X^n} = \nu^* B$, we see that the ideal $\bb \subseteq R^n$ is the reflexive hull of $\sO_X(-m B) R^n$.
Since $m B$ is Cartier at any codimension one point of $Z$ by an argument analogous to the proof of Lemma \ref{restriction of divisor} (1), the inclusion map $\sO_X(-m B) S' \hookrightarrow \bb S'$ is the identity at any codimension one point of $Z'$.
Composing with the inclusion $\sO_X(-m B) S \subseteq \sO_Z(-m B|_Z)$, we obtain the inclusion $\bb S' \subseteq L$ at any codimension one point of $Z'$.
It follows from the fact that $L$ is invertible and $Z'$ satisfies $(S_2)$ that $L=\bigcap_{x} L_x$ where $x$ runs through all codimension one points of $Z'$, which implies the desired inclusion $\bb S' \subseteq L$.
\end{proof}

As an intermediate step to prove $(\star)$, we show the inclusion 
\[
\bb S' \subseteq \mathcal{J}^{W+Z'}(X^n, A_{X^n} +Z', (\ba R^n)^\lambda \bb^{1-1/m}) S'.
\]
By Proposition \ref{non-lc restriction}, it is enough to show that if we set $\lambda' : = (m-1)/m$, $J : = \bb S'$ and $D : =\rho^* (m B|_Z)$ and if $U \subseteq Z'$ denotes the locus where $J=L$, then the assumptions (i), (ii) and (iii) in Proposition \ref{non-lc restriction} are satisfied.
We define $\Q$-Weil divisors $\Delta_{Z^n}$ and $A_{Z^n}$ on $Z^n$ as 
\[
\Delta_{Z^n} : = \rho^* \Delta|_Z + C_Z \textup{ and } A_{Z^n} : = \rho^* A|_Z + C_Z, 
\]
where $C_Z$ is the conductor divisor of $\rho$ on $Z^n$.
The assumption (i) is an immediate consequence of Claim 1 and  Proposition \ref{non-lc ideal twists} (3), because $\mathrm{Diff}_{Z^n}(A_{X^n}) =A_{Z^n} \sim - K_{Z^n}$ by Lemmas  \ref{Diff vs Cond1}, \ref{Diff vs Cond2} and \ref{restriction of divisor} and $D = m(A_{Z^n} - \Delta_{Z^n})$. 
Since $L$ is a principal ideal, $J_x \subseteq \m_{Z', x} L_x$ for all points $x \in Z' \setminus U$, from which the assumption (ii) follows. 
In order to verify the assumption (iii), we need the following claim. 
\begin{cln}
Let $V \subseteq X^n$ be the locus where $mB_{X^n}$ is Cartier. Then $U \subseteq V \cap Z $.
\end{cln}

\begin{proof}[Proof of Claim $2$]
Since $\bb$ satisfies $(S_2)$, $\bb \otimes_{R^n} S'$ is torsion-free, and therefore, we have an isomorphism $\bb \otimes_{R^n} S' \cong \bb S'=J$.  
If $x \in Z'$ is contained in $U$, then $\bb_x \otimes_{\sO_{X^n, x}} \sO_{Z',x} \cong J_x =L_x$ is an invertible $\sO_{Z', x}$-module, which implies that $\bb_x$ is an invertible $\sO_{X^n, x}$-module, that is, $m B_{X^n}$ is Cartier at $x$.
Thus, we obtain the assertion. 
\end{proof}

Let $\widetilde{U} : = V \cap Z'$ and $\widetilde{U}^n : = \pi^{-1}(\widetilde{U}) \subseteq Z^n$.
By Lemma \ref{J basic}, we have
\begin{align*}
&\mathcal{J}^{W+Z'}(X^n,A_{X^n} +Z', (\ba R^n)^{\lambda} \bb^{\lambda'})|_V \\
=& \mathcal{J}^{W|_V+ \widetilde{U}}(V, A_{X^n}|_V +\widetilde{U}, (\ba R^n)|_V^{\lambda} \bb|_V^{\lambda'}) \\
=& \mathcal{J}^{W|_V+ \widetilde{U}}(V, A_{X^n}|_V+\widetilde{U} + \frac{m-1}{m}(mB_{X^n}|_V) , (\ba R^n)|_V^{\lambda}) \\
=& \mathcal{J}^{W|_V+ \widetilde{U}}(V, \Delta_{X^n}|_V+ \widetilde{U} , (\ba R^n)|_V^{\lambda}) \otimes_{\sO_V} \sO_V(-mB_{X^n}|_V),
\end{align*}
where the second equality follows from the fact that $mB_{X^n}|_V$ is Cartier. 
Since the triple $(Z, \Delta|_Z, (\ba S)^\lambda)$ is slc and 
\[
\Delta_{Z^n}|_{\widetilde{U}^n} = \mathrm{Diff}_{Z^n}(\Delta_{X^n})|_{\widetilde{U}^n}= \mathrm{Diff}_{\widetilde{U}^n}(\Delta_{X^n}|_V)
\] 
by Lemmas \ref{Diff vs Cond1}, \ref{Diff vs Cond2} and \ref{restriction of divisor}, the triple $(\widetilde{U}^n, \mathrm{Diff}_{\widetilde{U}^n}(\Delta_{X^n}|_V), (\ba \sO_{\widetilde{U}^n} )^\lambda)$ is lc.
Noting that $m(K_V + \Delta_{X^n}|_V +\widetilde{U})$ is Cartier, we use inversion of adjunction for lc singularities \cite{Kaw}\footnote{Kawakita proved inversion of adjunction for lc pairs, but his proof works for triples.} 
to deduce that there exists an open subscheme $\widetilde{V} \subseteq V$ containing $\widetilde{U}$ such that $(\widetilde{V}, \Delta_{X^n}|_{\widetilde{V}} +\widetilde{U}, \ba|_{\widetilde{V}}^{\lambda})$ is lc, which is equivalent by Lemma \ref{J vs lc} (2) to saying that 
$\mathcal{J}^{W|_{\widetilde{V}} + \widetilde{U}}(\widetilde{V}, \Delta_{X^n}|_{\widetilde{V}} +\widetilde{U}, \ba|_{\widetilde{V}}^{\lambda}) = \sO_{\widetilde{V}}$.
Therefore, 
\[
\mathcal{J}^{W+ Z'}(X^n,A_{X^n}+Z', (\ba R^n)^{\lambda} \bb^{\lambda'})|_{\widetilde{V}} = \sO_{\widetilde{V}}(-mB_{X^n}|_{\widetilde{V}}),
\]
and it follows from Claim 2 that the assumption (iii) of Theorem \ref{non-lc restriction} is satisfied. 
Thus, we obtain the inclusion
\[
\bb S' \subseteq \mathcal{J}^{W+ Z'}(X^n, A_{X^n}+Z', (\ba R^n)^\lambda \bb^{1-1/m}) S'.
\] 

Finally, combining this inclusion with Proposition \ref{non-lc ideal twists} (1) yields that 
\[
\bb \subseteq \mathcal{J}^{W+ Z'}(X^n, A_{X^n}+Z', (\ba R^n)^\lambda \bb^{1-1/m}) + \bb \cap (h).
\]
Since $B_{X^n}$ has no common component with $Z'$, the ideal $\bb \cap (h)$ is contained in $h \bb$. 
By Nakayama's lemma, one has the desired inclusion $(\star)$, that is, 
\[
\bb=\mathcal{J}^{W+ Z'}(X^n, A_{X^n}+Z', (\ba R^n)^\lambda \bb^{1-1/m}). 
\]
\end{proof}

\begin{cor}\label{slc cut2}
In setting \ref{setting for inversion}, we further assume the condition  
\begin{enumerate}
\item[$(6')$] $R$ is normal. 
\end{enumerate}
If $(Z, \Delta|_Z, (\ba S)^{\lambda})$ is slc, then $(X,\Delta +Z, \ba^{\lambda})$ is valuatively lc.
\end{cor}

\begin{proof}
This is an immediate consequence of Theorem \ref{slc cut}. 
Since $X^n \cong X$, the assumption (6) in Theorem \ref{slc cut} is clearly satisfied. 
\end{proof}

\begin{cor}\label{slc cut3}
In setting \ref{setting for inversion}, we further assume that 
\begin{enumerate}
\item[$(6'')$] there exists an effective $\Q$-Weil divisor $\Theta$ such that $K_{X^n}+\Theta$ is $\Q$-Cartier and $\Theta \le \nu^*\Delta + C_X$, where $\nu:X^n \to X$ is the normalization of $X$ and $C_X$ is the conductor divisor of $\nu$ on $X^n$.
\end{enumerate}
If $(Z, \Delta|_Z, (\ba S)^{\lambda})$ is slc, then $(X,\Delta +Z, \ba^{\lambda})$ is valuatively slc. 
\end{cor}

\begin{proof}
As in the proof of Theorem \ref{slc cut}, we consider the following diagram.
\[
\xymatrix{
 X  & X^n \ar_-{\nu}[l] & \\
Z \ar@{^{(}->}^-{}[u]  & Z' \ar@{^{(}->}^-{}[u] \ar^-{\mu}[l] & Z^n=(Z')^n \ar^-{\pi}[l] \ar@/^18pt/[ll]^{\rho} \ar_-{f}[ul]
}\]
Since $\Diff_{Z^n}(\Theta) \le \Diff_{Z^n}(\nu^* \Delta +C_X) =\rho^*(\Delta|_Z) +C_Z$ by Lemmas \ref{Diff vs Cond1}, \ref{Diff vs Cond2} and \ref{restriction of divisor}, the pair $(Z^n, \Diff_{Z^n}(\Theta))$ is lc. 
We use inversion of adjunction for lc singularities \cite{Kaw} to deduce that $(X^n, \Theta + Z')$ is lc near $Z'$.
It then follows from \cite[Theorem 3.4]{Ale} that $Z'$ satisfies $(S_2)$.
Now we apply Theorem \ref{slc cut} to obtain the result. 
\end{proof}

As a corollary, we obtain results on deformations of slc singularities. 

\begin{cor}\label{slc local}
With notation as in Setting \ref{local setting}, let $x \in X$ be a closed point and $\mathcal{Z} \subseteq \mathcal{X}$ be an irreducible closed subscheme such that $(\mathcal{X}, i, \mathcal{Z}, j)$ is a deformation of the pair $(X, \{x\}_{\mathrm{red}})$ over $T$ with reference point $t$. 
Let $y$ be the generic point of $\mathcal{Z}$, which lies in the generic fiber $\mathcal{X}_\eta$. 
We assume that the following conditions are satisfied: 

\begin{enumerate}[label=\textup{(\arabic*)}]
\item $T$ is a smooth curve, 
\item $K_{X}+\mathcal{D}|_X$ is $\Q$-Cartier at $x$.
\end{enumerate}
If $(X, \mathcal{D}|_X, (\ba \sO_X)^\lambda)$ is lc at $x$, then $(\mathcal{X}_\eta, \mathcal{D}_\eta, \ba_\eta^\lambda)$ is valuatively lc at $y$.
If we further assume the condition 
\begin{enumerate}
\item[$(3)$] the closed fiber $\mathcal{X}^n_t$ of the normalization $\mathcal{X}^n$ of $\mathcal{X}$ satisfies $(S_2)$, 
\end{enumerate}
then the slc case also holds, that is, if $(X, \mathcal{D}|_X, (\ba \sO_X)^\lambda)$ is slc at $x$, then $(\mathcal{X}_\eta, \mathcal{D}_\eta, \ba_\eta^\lambda)$ is valuatively slc at $y$.
\end{cor}

\begin{proof}
It follows from Theorem \ref{slc cut}, Corollary \ref{slc cut2} and Corollary \ref{slc cut3} that $(\mathcal{X}, \mathcal{D}, \ba^\lambda)$ is valuatively (s)lc at $x$.
Since $y$ is a generalization of $x$, the triple $(\mathcal{X}, \mathcal{D}, \ba^\lambda)$ is valuatively (s)lc at $y$ by Remark \ref{remark on slc}, which completes the proof.
\end{proof}

\begin{rem}
Koll\'ar points out in a draft of his book \cite[Theorem 5.33]{Kol2}, whose method can be traced back to his joint work \cite[Corollary 5.5]{KSB} with Shepherd-Barron, that the slc case of Corollary \ref{slc local} holds without the condition (3). 
However, since his proof heavily depends on the existence of lc modifications, we believe that our proof, which uses only the vanishing theorem for lc singularities, is of independent interest. 
\end{rem}

\begin{cor}\label{slc proper}
With notation as in Setting \ref{local setting}, we further assume that the following conditions are all satisfied:
\begin{enumerate}[label=\textup{(\arabic*)}]
\item $T$ is a smooth curve, 
\item $K_{X}+\mathcal{D}|_X$ is $\Q$-Cartier,
\item $\mathcal{X}$ is proper over $T$.
\end{enumerate} 
If $(X, \mathcal{D}|_X, (\ba \sO_X)^\lambda)$ is lc, then $(\mathcal{X}_{\eta}, \mathcal{D}_{\eta}, \ba_{\eta}^\lambda)$ is valuatively lc. 
If we further assume the condition 
\begin{enumerate}
\item[$(4)$] the closed fiber $\mathcal{X}^n_t$ of the normalization $\mathcal{X}^n$ of $\mathcal{X}$ satisfies $(S_2)$, 
\end{enumerate}
then the slc case also holds, that is, if $(X, \mathcal{D}|_X, (\ba \sO_X)^\lambda)$ is slc, then $(\mathcal{X}_\eta, \mathcal{D}_\eta, \ba_\eta^\lambda)$ is valuatively slc.
\end{cor}

\begin{proof}
Since the structure map $\mathcal{X} \to T$ is a closed map, it follows from an argument similar to the proof of Corollary \ref{slc local} that $(\mathcal{X}, \mathcal{D}, \ba^\lambda)$ is valuatively (s)lc near $\mathcal{X}_{\eta}$, which implies the assertion. 
\end{proof}

\begin{rem}
In Corollary \ref{slc local} (resp.~Corollary \ref{slc proper}), the condition (3) (resp.~(4)) is satisfied for example if one of the following holds: 
\begin{enumerate}[label=\textup{(\alph*)}]
  \item $\mathcal{X}$ is normal, or 
  \item there exists an effective $\Q$-Weil divisor $\Theta$ on $\mathcal{X}^n$ such that $K_{\mathcal{X}^n}+\Theta$ is $\Q$-Cartier and $\Theta \le \nu^*\mathcal{D} + C_{\mathcal{X}}$, where $\nu: \mathcal{X}^n \to \mathcal{X}$ is the normalization of $\mathcal{X}$ and $C_{\mathcal{X}}$ is the conductor divisor of $\nu$ on $\mathcal{X}^n$.
  \end{enumerate}
This follows from arguments similar to the proofs of Corollaries \ref{slc cut2} and \ref{slc cut3}. 
\end{rem}

Finally, we show that slc singularities are invariant under small deformations if the total space is normal and the nearby fibers are $\Q$-Gorenstein. 

\begin{cor}\label{general lc fibers}
Let $T$ be a smooth curve over an algebraically closed field $k$ of characteristic zero and $(\mathcal{X}, \mathcal{D}, \ba^{\lambda}) \to T$ be a proper flat family of triples over $T$, where $\mathcal{D}$ is an effective $\Q$-Weil divisor on a normal variety $\mathcal{X}$ over $k$, $\ba \subseteq \sO_X$ is a nonzero coherent ideal sheaf and $\lambda > 0$ is a real number. 
\begin{enumerate}[label=\textup{(\arabic*)}]
\item 
Suppose that $k$ is an uncountable algebraically closed field. 
If some closed fiber $(\mathcal{X}_{t_0}, \mathcal{D}_{t_0}, (\ba \sO_{\mathcal{X}_{t_0}})^{\lambda})$ is slc and if a general closed fiber $(\mathcal{X}_t, \mathcal{D}_t)$ is log $\Q$-Gorenstein, then $(\mathcal{X}_{t}, \mathcal{D}_{t}, (\ba \sO_{\mathcal{X}_{t}})^{\lambda})$ is lc.
\item If some closed fiber $(\mathcal{X}_{t_0}, \mathcal{D}_{t_0}, (\ba \sO_{\mathcal{X}_{t_0}})^{\lambda})$ is two-dimensional lc, then so is a general closed fiber $(\mathcal{X}_{t}, \mathcal{D}_{t}, (\ba \sO_{\mathcal{X}_{t}})^{\lambda})$. 
\end{enumerate}
\end{cor}

\begin{proof}
In both cases, it suffices to show that the generic fiber $(X_{\eta}, D_{\eta}, (\ba \sO_{X_\eta})^{\lambda})$ is lc. 
In (1), since  $K_{X_{\eta}}+D_{\eta}$ is $\Q$-Cartier by \cite[Remark 2.15]{ST},  it follows from Corollary \ref{slc proper} and Remark \ref{remark on val lc} that $(X_{\eta}, D_{\eta}, (\ba \sO_{X_\eta})^{\lambda})$ is lc. 
In (2), we deduce from Corollary \ref{slc proper} that $(X_{\eta}, D_{\eta}, (\ba \sO_{X_\eta})^{\lambda})$ is two-dimensional valuatively lc, which implies by Lemma \ref{2-dim val lc} that $(X_{\eta}, D_{\eta}, (\ba \sO_{X_\eta})^{\lambda})$ is lc. 
\end{proof}

\begin{rem}
Using plurigenera defined for normal isolated singularities, Ishii \cite{ish} proved the isolated singularities case of 
Corollary \ref{general lc fibers} (1). 
She also showed (the no boundary case of) Corollary \ref{general lc fibers} (2), combining results of \cite{ish} and  \cite{ish2}. 
Thus, Corollary \ref{general lc fibers} gives a generalization and an alternative proof of her results. 
\end{rem}

\appendix

\section{Some background material on AC divisors}\label{appendix}

\subsection{Notation}\label{AC Notation}
Throughout this subsection, we assume that $X$ is an excellent reduced scheme satisfying the $(S_2)$-condition.
Let $\mathcal{K}_X$ denote the sheaf of total quotients of $X$.

First we recall the definition of AC divisors.
The reader is referred to \cite[Subsection 2.1]{MS} and \cite[Section 16]{Kol92} for more details.

\begin{defn}
An \emph{AC divisor} (or \emph{almost Cartier divisor}) is a coherent submodule $\mathcal{F} \subseteq \mathcal{K}_X$ satisfying the following two conditions:
\begin{enumerate}
\item $\mathcal{F}$  satisfies $(S_2)$ and
\item $\mathcal{F}_x$ is an invertible $\sO_{X,x}$-module for each point $x \in X$ of codimension$\le 1$.
\end{enumerate}
\end{defn}

AC divisors form an additive group via tensor product up to $S_2$-ification (\cite[Section 5.10]{EGA IV-2}), which is denoted by $\WSh(X)$.
Let $D$ denote an AC divisor $\mathcal{F} \subseteq \mathcal{K}_X$.
We say that $D$ is \emph{effective} if $\sO_X \subseteq \mathcal{F}$.  
We also say that $D$ is \emph{Cartier} at a point $x \in X$ if $\mathcal{F}$ is invertible at $x$, and that $D$ is Cartier if $D$ is Cartier at all points of $X$. 
Note that the set of all Cartier AC divisors coincides with the image of the injective group homomorphism
\[
\CDiv(X) \hookrightarrow \WSh(X); \ E \mapsto \sO_X(E),
\]
where $\CDiv(X) = H^0(X, \mathcal{K}_X^{*}/\sO_{X}^{*})$ is the set of all Cartier divisors.
We say that two AC divisors $D_1$ and $D_2$ are \emph{linearly equivalent} if $D_1-D_2$ is contained in the image of $\mathrm{Pr}(X) \subseteq \CDiv(X) \to \WSh(X)$, where $\mathrm{Pr}(X)$ denotes the set of all principal divisors.

By a \emph{$\Q$-AC divisor}, we mean an element of 
\[
\WSh_{\Q}(X) : = \WSh(X) \otimes_\Z \Q.
\]
We say that a $\Q$-AC divisor $\Delta \in \WSh_{\Q}(X)$ is \emph{effective} (resp.~\emph{$\Q$-Cartier} at a point $x \in X$, \emph{$\Q$-Cartier}) if $\Delta = D \otimes \lambda$ for some effective (resp.~Cartier at $x$, Cartier) AC divisor $D \in \WSh(X)$ and some nonnegative rational number $\lambda$.

The \emph{support} of an AC divisor $\mathcal{F} \subseteq \mathcal{K}_X$ is the closed subset consisting of all points $x \in X$ such that $\mathcal{F}_x \neq \sO_{X,x}$ as a submodule of $\mathcal{K}_{X,x}$.

\begin{lem}\label{AC support}
The support of an AC divisor is of pure codimension one if it is not empty.
\end{lem}

\begin{proof}
Let $D$ denote an AC divisor $\mathcal{F} \subseteq \mathcal{K}_X$ whose support is not empty. 
Assume to the contrary that there exists an irreducible component $Z$ of the support $\Supp D$ of $D$ with codimension$\ge 2$.
After shrinking $X$, we may assume that $\Supp D=Z$.
Let $i : U \hookrightarrow X$ be an open immersion from $U : = X \setminus Z$, and then it follows from Lemma \ref{S2} that $\mathcal{F} = i_* i^* \mathcal{F} = i_* \sO_{U} = \sO_X$. 
This is a contradiction to the assumption that $\Supp D \neq \emptyset$.
\end{proof}

For a Weil divisor $E$ contained in $\WDiv^*(X)$, since the submodule $\sO_X(E) \subseteq \mathcal{K}_X$ is an AC divisor, we obtain the injective group homomorphism
\[
\WDiv^*(X) \to \WSh(X); \ E \mapsto \sO_X(E). 
\]
Its image is the subgroup $\WSh^*(X)$ of $\WSh(X)$ consisting of all AC divisors whose supports contain no codimension one singular points of $X$.
The situation is summarized in the following commutative diagram, which is cartesian.
\[
\xymatrix{
\CDiv(X) \ar@{^{(}->}[rr] & & \WSh(X) \\
\CDiv^*(X) \ar@{^{(}->}[u] \ar@{^{(}->}[r] & \WDiv^*(X) \ar^-{\sim}[r] & \WSh^*(X) \ar@{^{(}->}[u]
}\]
Since $\WSh^*(X)$ is a free $\Z$-module, the natural map 
\[
\WSh^*(X) \to \WSh^*_{\Q}(X) : = \WSh^*(X) \otimes_\Z \Q \subseteq \WSh_\Q (X)
\] 
is injective.
Let $\Delta$ be a $\Q$-AC divisor contained in $\WSh^*_{\Q}(X)$.
Then there exists an integer $m \ge 1$ such that $m\Delta$ is integral, that is, $m \Delta \in \WSh^*(X)$.
We define the support $\Supp \Delta$ of $\Delta$ as $\Supp m \Delta$.
This is independent of the choice of $m$ by the following lemma.

\begin{lem}\label{AC support 2}
The support of a Weil divisor $E \in \WDiv^*(X)$ coincides with that of the AC divisor $\sO_X(E)$.
In particular, for every AC divisor $D \in \WSh^*(X)$ and every integer $n \ge 1$, we have $\Supp(D)=\Supp(nD)$.
\end{lem}

\begin{proof}
It immediately follows from Lemma \ref{AC support}.
\end{proof}

\begin{rem}
There is an example of $X$ such that the natural map
\[
\WSh(X) \to \WSh_{\Q}(X)
\]
is not injective (see \cite[(16.1.2)]{Kol92}).

We also have an example of an AC divisor $D$ such that $\Supp D \neq \Supp n D$ for an integer $n \ge 1$.  
If we set $X : = \Spec \C[x,y]/(y^2 - x^2 +x^3 )$ and $D:= (y/x) \sO_X$, then the origin $(0, 0) \in X$ is contained in the support of $D$ but not in that of $2D=(x-1) \sO_X$. 
\end{rem}

\subsection{Differents of AC divisors}\label{Subsection AC-Diff}

In this subsection, we recall the definition of the different of a $\Q$-AC divisor and prove some basic results used in subsection \ref{Subsection Diff}.

Throughout this subsection, we fix an excellent scheme $S$ admitting a dualizing complex $\omega_S^{\bullet}$, every scheme is assumed to be separated and of finite type over $S$ and every morphism is assumed to be an $S$-morphism.
Moreover, given a scheme $X$, we always choose $\omega_X^{\bullet} : = \pi_X^{!} \omega_S^{\bullet}$ as a 
a dualizing complex of $X$, where $\pi_X: X \to S$ is the structure morphism, and $\omega_X$ always denotes the canonical sheaf associated to $\omega_X^{\bullet}$.
The trace map of a finite surjective morphism $f: Y \to X$
is denoted by $\mathrm{Tr}_f: f_* \omega_Y \to \omega_X$ (see Subsection \ref{duality} for its definition). 

\begin{lem}\label{Tr-transpose}
Let $f: Y \to X$ be a finite birational morphism of reduced schemes.
Then the following hold.
\begin{enumerate}[label=$(\arabic*)$]
\item For a morphism $\alpha_X: \omega_X \to \mathcal{K}_X$, there exists a unique morphism $\alpha_Y : \omega_Y \to \mathcal{K}_Y$ such that the following diagram commutes.
\begin{equation}\label{Tr-transpose diagram}
\vcenter{
\xymatrix{
f_* \omega_Y \ar^-{\alpha_Y}[r] \ar_-{\mathrm{Tr}_f}[d] & f_* \mathcal{K}_Y \\
\omega_X \ar_-{\alpha_X}[r] & \mathcal{K}_X \ar_-{\rotatebox{90}{$\sim$}}^-{\theta_f}[u]
}},
\end{equation}
where $\theta_f: \mathcal{K}_X \xrightarrow{\sim} f_* \mathcal{K}_Y$ is the canonical isomorphism.
\item For a morphism $\alpha_Y : \omega_Y \to \mathcal{K}_Y$, there exists a unique morphism $\alpha_X : \omega_X \to \mathcal{K}_X$ such that the diagram \eqref{Tr-transpose diagram} commutes.
\end{enumerate}
\end{lem}

\begin{proof}
Take an open subscheme $i: U \hookrightarrow X$ containing all generic points of $X$ such that $V: = f^{-1}(U) \to U$ is an isomorphism. 
Since $i_* \mathcal{K}_U = \mathcal{K}_X$, we have
\[
\Hom_X(\omega_X, \mathcal{K}_X) \cong \Hom_U(\omega_U , \mathcal{K}_U).
\]
Similarly, 
\[
\Hom_Y(\omega_Y, \mathcal{K}_Y) \cong \Hom_{V}(\omega_{V} , \mathcal{K}_{V}).
\]
Combining these isomorphisms with the commutative diagram \eqref{Trace open diagram} in Subsection \ref{duality}, after replacing $X$ by $U$, we may assume that $f$ is an isomorphism.
In this case, the assertion is obvious because $\mathrm{Tr}_f$ is an isomorphism.
\end{proof}

Let $\omega_X^{\bullet}$ be a dualizing complex of $X$ and $\omega_X$ be the canonical sheaf associated to $\omega_X^{\bullet}$. 
A \emph{canonical AC divisor} on $X$ associated to $\omega_X^{\bullet}$ is an AC divisor $\mathcal{F} \subseteq \mathcal{K}_X$ such that $\mathcal{F} \cong \omega_X$ as $\sO_X$-modules.
The reader is referred to Lemma \ref{canonical AC divisor2} below for sufficient conditions for $X$ to admit a canonical AC divisor.

\begin{setting}\label{setup for AC-Diff}
Let $\mathcal{A}: = (Y, W, W', i, \mu, f, \Gamma)$ be a tuple satisfying the following conditions. 
\begin{enumerate}
\item $Y$ is an excellent reduced $(S_2)$ and $(G_1)$ scheme admitting a canonical AC divisor associated to $\omega_Y^{\bullet} := \pi_Y^! \omega_S^{\bullet}$, where $\pi_Y:Y \to S$ is the structure morphism. 
\item $i: W \hookrightarrow Y$ is the closed immersion from a reduced closed subscheme $W$ whose generic points are codimension one regular points of $Y$. 
\item $\mu: W' \to W$ is a finite birational morphism from a reduced $(S_2)$ and $(G_1)$ scheme $W'$ and $f: = i \circ \mu : W' \to Y$ is the composite of $i$ and $\mu$. 
\[
\xymatrix{
Y & \\
W \ar^-{i}[u] & W' \ar^-{\mu}[l] \ar_-{f}[ul]
}\]
\item $\Gamma \in \WSh^*_{\Q}(Y)$ is a $\Q$-AC divisor on $Y$ such that the support of $\Gamma$ has no common components with $W$ and the $\Q$-AC divisor $K_Y+ \Gamma+W$ is $\Q$-Cartier at every codimension one point $w$ of $W$.
\end{enumerate}
\end{setting}

Suppose that $\mathcal{A} : = (Y,W, W', i, \mu, f, \Gamma)$ is a tuple as in Setting \ref{setup for AC-Diff}.
Since $Y$ admits a canonical AC divisor $K_Y$, we have an inclusion $\alpha_Y: \omega_Y \hookrightarrow \mathcal{K}_Y$ whose image coincides with $K_Y$. 
By Lemma \ref{canonical AC divisor2} (i), $W'$ also admits a canonical AC divisor $K_{W'}$ and let $\alpha_{W'}: \omega_{W'} \hookrightarrow \mathcal{K}_{W'}$ be the corresponding inclusion. 
Take an integer $m \ge 1$ such that $m \Gamma \in \WSh^*(X)$ and $m(K_Y+W)+m \Gamma$ is Cartier at any codimension one points of $W$, and let $\mathcal{F} :=\sO_Y(m(K_Y+W)+m\Gamma) \subseteq \mathcal{K}_{Y}$. 
We will define the morphism
\[
\beta_{\mathcal{A}}(\alpha_Y, \alpha_{W'},m ): f^* \mathcal{F} \to \mathcal{K}_{W'}.
\]

Take an open subscheme $V \subseteq Y$ such that $V$ is regular, $\Gamma|_V=0$, $U : = W \cap V$ is regular and $U$ contains all generic points of $W$.
Let $u$ and $v$ be natural open immersions such that the following diagram commutes:
\[
\xymatrix{
W \ar@{^{(}->}[r]^-{i} & Y \\
U \ar@{^{(}->}[u]^-{u} \ar[r] & V. \ar@{^{(}->}[u]^-{v}
}\]
Then we define the morphism $\beta_{\mathcal{A}}^V(\alpha_Y, \alpha_{W'}, m ) : \mathcal{F}|_U \to \mathcal{K}_U$ as 
\[
\mathcal{F}|_U =\sO_V(m (K_Y|_V +U))|_U \xrightarrow{(\gamma_j^m)^{-1} \circ (\alpha_Y^m)^{-1}} (\omega_V(U)|_U)^{\otimes m} \xrightarrow{\mathrm{Res}_{V/U}^m} \omega_U^m \xrightarrow{\alpha_W^m \circ \gamma_i^m} \mathcal{K}_U,
\]
where $\mathrm{Res}_{V/U}$, $\gamma_i$ and $\gamma_j$ are isomorphisms defined in Subsection \ref{duality} and $\alpha_W : \omega_W \to \mathcal{K}_W$ is the morphism induced by $\alpha_{W'}$ as in Lemma \ref{Tr-transpose}. 
Pulling back this morphism to $U' : = f^{-1}(U) \subseteq W'$ and taking the $(u')^*$-$(u')_*$ adjoint, where $u' : U' \hookrightarrow W'$ is the open immersion, we obtain a morphism $f^* \mathcal{F} \to \mathcal{K}_{W'}$.
Since this morphism is independent of the choice of $V$ by the diagrams \eqref{open composition diagram} and \eqref{Res open} in Subsection \ref{duality}, we write this morphism by 
\[
\beta_{\mathcal{A}}(\alpha_Y, \alpha_{W'}, m): f^* \mathcal{F} \to \mathcal{K}_{W'}.
\]

Let $E =E_{\mathcal{A}}(\alpha_Y, \alpha_{W'}, m) \in \WSh(W')$ denote the AC divisor defined by the reflexive hull of the image of $\beta_{\mathcal{A}}(\alpha_Y, \alpha_{W'}, m)$.
Then the \emph{different} of $\Gamma$ on $W'$ is defined as 
\[
\ACDiff_{W'}(\Gamma) : = (E - m K_{W'}) \otimes_{\Z} \frac{1}{m} \in \WSh_\Q(W'). 
\]

\begin{lem}\label{basic for AC-Diff}
With the above notation, the following holds.
\begin{enumerate}[label=$(\arabic*)$]
\item The $\Q$-AC divisor $\ACDiff_{W'}(\Gamma)$ is independent of the choice of $\alpha_Y$, $\alpha_{W'}$ and $m$.
\item Taking differents is compatible with open immersions, that is, for an open subscheme $Y^{\circ} \subseteq Y$, we have 
\[
\ACDiff_{(W')^{\circ}}(\Gamma|_{Y^{\circ}})=\ACDiff_{W'}(\Gamma)|_{(W')^{\circ}},
\]
where $(W')^{\circ} \subseteq W'$ is the pullback of $Y^{\circ}$ to $W'$.
\item If $D \in \CDiv_{\Q}^*(Y)$ is a $\Q$-Cartier divisor on $Y$ whose support does not contain any generic points of $W$, then
\[
\ACDiff_{W'}(\Gamma+\widetilde{D}) = \ACDiff_{W'}(\Gamma ) + \widetilde{f^*D},
\]
where $\widetilde{D}$ and $\widetilde{f^*D}$ are the $\Q$-AC divisors corresponding to the $\Q$-Cartier divisors $D$ and $f^*D$, respectively.
\end{enumerate}
\end{lem}

\begin{proof}
(1) and (2) are obvious.
For (3), after shrinking $Y$ if necessary, we can write $D=D_1-D_2$, where $D_1, D_2 \in \CDiv^*_{\Q}(Y)$ are effective $\Q$-Cartier divisors whose supports contain no generic points of $W$. 
Therefore, it suffices to show the assertion when $D$ is effective.

Take $m, \alpha_Y, \alpha_{W}$ and $K_{W}$ as in the discussion preceding this lemma. 
We further assume that $m D$ is Cartier.
When we write $\mathcal{A}' : = (Y, W, W', i, \mu, f, \Gamma + \widetilde{D})$, the following diagram 
\[
\xymatrix{
&\hspace*{3.5em} f^* \sO_Y(m (K_Y+ W) +m \Gamma) \ar@{_{(}->}_-{\textup{nat}}[ld]  \ar^-{\beta_{\mathcal{A}}}[dd]   \\
f^* \sO_Y(m (K_Y+ W) +m \Gamma) \otimes \sO_{W'}(m f^*(D)) \ar_-{\textup{\rotatebox{90}{$\sim$}}}[d] &  \\
f^* \sO_Y(m (K_Y+ W) +m (\Gamma+\widetilde{D})) \ar^-{\beta_{\mathcal{A'}}}[r] & \mathcal{K}_{W'} &
}
\]
commutes, because the morphism $\beta_{\mathcal{A'}}: =\beta_{\mathcal{A}'}(\alpha_Y,\alpha_{W'}, m)$ is generically the same as $\beta_{\mathcal{A}}: =\beta_{\mathcal{A}}(\alpha_Y,\alpha_{W'}, m)$.
Thus, 
\[
E_{\mathcal{A'}}(\alpha_Y,\alpha_{W'}, m) =E_{\mathcal{A}}(\alpha_Y,\alpha_{W'}, m) + m \widetilde{f^*D},
\]
which implies the desired result. 
\end{proof}

\begin{lem}\label{restriction of AC-divisor}
Suppose that $Y$ is a scheme satisfying the condition $(1)$ in Setting \ref{setup for AC-Diff} and $i: W \hookrightarrow Y$ be a closed immersion satisfying the condition $(2)$.  
We further assume that $W$ is a Cartier divisor (that is, $W \in \CDiv^*(Y)$) satisfying $(S_2)$ and $(G_1)$. 
Take a $\Q$-Weil divisor $\Delta \in \WDiv^*_{\Q}(Y)$ whose support contains neither any generic points of $W$ nor any singular codimension one points of $W$, and let $\Gamma \in \WSh^*_{\Q}(Y)$ be the corresponding $\Q$-AC divisor. 

\begin{enumerate}[label=$(\arabic*)$]
\item The tuple $(Y, W, W, i, \mathrm{id}_{W},i, \Gamma)$ satisfies all the conditions in Setting \ref{setup for AC-Diff}.
\item Let $\Delta|_W \in \WDiv_\Q^*(W)$ be the restriction of $\Delta$ to the Cartier divisor $W$ (see Lemma \ref{restriction of divisor} for the definition).
Then the $\Q$-Weil divisor $\Delta|_W \in \WDiv_{\Q}^*(W)$ corresponds to the $\Q$-AC divisor $\ACDiff_W(\Delta) \in \WSh_{\Q}(W)$.
\end{enumerate}
\end{lem}

\begin{proof}
(1) It is enough to verify the condition (4) in Setting \ref{setup for AC-Diff}.
Let $w \in W$ be a codimension one point.
If $w$ is a singular point of $W$, then $K_Y+W+\Delta = K_Y+W$ around $w$, which is Cartier since $W$ is Cartier and satisfies $(G_1)$.
If $w$ is a regular point of $W$, then $Y$ is also regular at $w$ and in particular $K_Y+W+\Delta$ is $\Q$-Cartier at $w$.

(2) Take $m, \alpha_Y, \alpha_{W}$ and $K_{W}$ as in the discussion preceding Lemma \ref{basic for AC-Diff}. 
We will show that the AC-divisor $E_{\mathcal{A}}(\alpha_Y, \alpha_{W}, m)-m K_W \in \WSh(W)$ coincides with the Weil divisor $m \Delta|_W \in \WDiv^*(W)$.
By Lemma \ref{S2} (3), it is enough to show the assertion after shrinking $Y$ around an arbitrary codimension one point $w$ of $W$.

First, we consider the case where $w$ is a singular point of $W$.
After shrinking $Y$, we may assume that $\Delta=0$ and $Y$, $W$ are Gorenstein.
Since the Poincar\'{e} residue map $\mathrm{Res}_{Y/W}: \omega_Y(W)|_W \to \omega_W$ is isomorphic (Lemma \ref{residue isom} below), it induces the isomorphism
\[
\phi: (\sO_Y(m(K_Y+W))|_W) \cong (\omega_Y(W)|_W)^{m} \xrightarrow{(\mathrm{Res}_{Y/W})^m } (\omega_W)^m \cong \sO_W(m K_W) \subseteq \mathcal{K}_W.
\]
Since the Poincar\'{e} residue maps is compatible with open immersions (see the commutative diagram \eqref{Res open}), we have $\phi=\beta_{\mathcal{A}}(\alpha_Y, \alpha_W, m)$.
Therefore,  $E_{\mathcal{A}}(\alpha_Y,\alpha_W, m) =m K_W$. 

Next, we consider the case where $w$ is a regular point of $W$.
After shrinking $Y$, we may assume that $Y$ and $W$ are regular. 
Since $\Delta$ is $\Q$-Cartier, we can reduce to the case where $\Delta=0$ by applying Lemma \ref{basic for AC-Diff} (3).  
Then we obtain the equality $E_{\mathcal{A}}(\alpha_Y,\alpha_W, m) =m K_W$ as in the first case. 
\end{proof}

\begin{lem-defn}\label{Definition of Diff}
Let $(Y, W, W', i, \mu, f, \Delta)$ be as in Setting \ref{setup for Diff} and $\Gamma \in \WSh^*_{\Q}(Y)$ be the $\Q$-AC divisor corresponding to the $\Q$-Weil divisor $\Delta \in \WDiv^*_\Q(Y)$.
Then the following hold.
\begin{enumerate}
\item $\mathcal{A} : = (Y, W, W', i, \mu, f, \Gamma)$ satisfies all the conditions in Setting \ref{setup for AC-Diff}.
\item The $\Q$-AC divisor $\ACDiff_{W'}(\Gamma) \in \WSh_\Q(W')$ is contained in $\WSh^*_{\Q}(W')$.
\end{enumerate}
We define the \emph{different} $\Diff_{W'}(\Delta) \in \WDiv_{\Q}^*(W')$ of $\Delta$ on $W'$ as the $\Q$-Weil divisor corresponding to the $\Q$-AC divisor $\ACDiff_{W'}(\Gamma) \in \WSh^*_{\Q}(W')$.
\end{lem-defn}

\begin{proof}
(1) is obvious.
For (2), take $m, \alpha_Y, \alpha_{W'}$ and $K_{W'}$ as in the discussion preceding Lemma \ref{basic for AC-Diff}.  
It is enough to show that the AC-divisor $E_{\mathcal{A}}(\alpha_Y, \alpha_{W'}, m)-m K_{W'}$ is contained in $\WSh^*(W')$.
Take a codimension one singular point $w'$ of $W'$.
After shrinking $Y$ around $f(w')$, we may assume that $\Delta=0$ and $W$ is Cartier.
Then the equality $E_{\mathcal{A}}(\alpha_Y, \alpha_{W'}, m)-m K_{W'} =0$ can be shown in a way similar to the proof of Lemma \ref{restriction of AC-divisor}. 
\end{proof}

\begin{lem}\label{AC-Diff vs Cond1}
Let $ \mathcal{A}: =(Y, W, W', i, \mu, f, \Gamma)$ be as in Setting \ref{setup for AC-Diff} and $\pi: W^n = (W')^n \to W'$ be the normalization of $W'$.
Then 
\[
\ACDiff_{W^n}(\Gamma) = \pi^* \ACDiff_{W'}(\Gamma) + C_{W'},
\]
where $C_{W'}$ denotes the conductor divisor of $\pi$ on $W^n=(W')^n$. 
\end{lem}

\begin{proof}
We first note that the tuple $\mathcal{A}' : =(Y, W, W^n, i, \rho : = \mu \circ \pi, g : = f \circ \pi, \Delta)$ satisfies the conditions in Setting \ref{setup for AC-Diff}.
\[
\xymatrix{
 Y  & & \\
W \ar@{^{(}->}^-{i}[u]  & W' \ar^-{f}[ul] \ar^-{\mu}[l] & W^n \ar^-{\pi}[l] \ar@/^18pt/[ll]^{\rho} \ar_-{g}[ull]
}\]
After shrinking $Y$, we may assume that $K_Y+W+\Delta$ is $\Q$-Cartier.
Take an integer $m \ge 1$ such that $m \Gamma \in \WSh^{*}(Y)$ and $m(K_Y+W)+m\Gamma$ is Cartier. 
Since $Y$ admits a canonical AC divisor $K_Y$, we have the corresponding inclusion $\alpha_Y:\omega_Y \hookrightarrow \mathcal{K}_Y$. 
Let $\alpha_{W'} : \omega_{W'} \hookrightarrow \mathcal{K}_{W'}$ and $\alpha_{W^n} : \omega_{W^n} \hookrightarrow \mathcal{K}_{W^n}$ be inclusions such that a diagram involving $\alpha_{W'}$ and $\alpha_{W^n}$, similar to \eqref{Tr-transpose diagram}, commutes and 
let $K_{W'}$ and $K_{W^n}$ be the corresponding canonical AC divisors.
By the choice of $\alpha_{W'}$ and $\alpha_{W^n}$, we have $K_{W^n}=\pi^* K_{W'} - C_{W'}$.

Take an open subscheme $V \subseteq Y$ such that $V$ is regular, $\Gamma|_V=0$, $U : = W \cap V$ is regular and $U$ contains all generic points of $W$.
We also take an inclusion $\alpha_W : \omega_W \hookrightarrow \mathcal{K}_W$ such that a diagram involving $\alpha_{W}$ and $\alpha_{W'}$, similar to \eqref{Tr-transpose diagram}, commutes. 
It then follows from the equality $\mathrm{Tr}_{\mu} \circ \mu_* \mathrm{Tr}_{\pi} = \mathrm{Tr}_{\rho}$, which is a consequence of the diagram \eqref{Trace composition diagram},  that a similar diagram involving $\alpha_W$ and $\alpha_{W^n}$ also commutes.
Therefore,  
\[
\beta^V_{\mathcal{A}}(\alpha_Y,\alpha_{W'}, m)= \beta^V_{\mathcal{A'}}(\alpha_Y,\alpha_{W^n}, m), 
\]
which implies that the following diagram
\[
\xymatrix{
& g^* \sO_Y(m (K_Y+ W) +m \Gamma) \ar^-{\pi^* \beta_{\mathcal{A}}}[dl] \ar_-{\beta_{\mathcal{A'}}}[dr]& \\
\pi^* \mathcal{K}_{W'} \ar_-{\theta_{\pi}^{\mathrm{adj}}}[rr] && \mathcal{K}_{W^n} 
}
\]
commutes, where $\beta_{\mathcal{A}} = \beta_{\mathcal{A}}(\alpha_Y,\alpha_{W'}, m)$, $\beta_{\mathcal{A'}} = \beta_{\mathcal{A'}}(\alpha_Y,\alpha_{W^n}, m)$ and $\theta_{\pi}^{\mathrm{adj}}$ is the natural isomorphism.
Taking into account that $\sO_Y(m (K_Y+ W) +m \Gamma)$ is invertible, we have 
\begin{eqnarray*}
\sO_{W^n}(E_{\mathcal{A'}}(\alpha_{Y}, \alpha_{W^n}, m)) &=& \mathrm{Image}(\beta_{\mathcal{A}'}) \\
&=& \theta_{\pi}^{\mathrm{adj}}(\mathrm{Image}(\pi^* \beta_{\mathcal{A}})) \\
&=& \theta_{\pi}^{\mathrm{adj}}(\pi^* \sO_{W'}(E_{\mathcal{A}}(\alpha_Y, \alpha_{W'}, m)) \\
&=& \sO_{W^n}(\pi^*E_{\mathcal{A}}(\alpha_Y, \alpha_{W'}, m))
\end{eqnarray*}
 as submodules of $\mathcal{K}_{W^n}$. 
Thus, 
\begin{eqnarray*}
m(K_{W^n} + \ACDiff_{W^n}(\Gamma)) &=& E_{\mathcal{A'}}(\alpha_{Y}, \alpha_{W^n}, m) \\
&=& \pi^* E_{\mathcal{A}}(\alpha_Y, \alpha_{W'}, m) \\
&=& \pi^*( mK_{W'} + m \ACDiff_{W'}(\Gamma)) \\
&=& m(K_{W^n} +(\pi^* \ACDiff_{W'}(\Gamma) + C_{W'})),
\end{eqnarray*}
which completes the proof.
\end{proof}

\begin{lem}\label{AC-Diff vs Cond2}
Let $\mathcal{A} : = (Y, W, W', i, \mu, f, \Gamma)$ be as in Setting \ref{setup for AC-Diff} such that $f: W' \to Y$ factors through the normalization $\nu: Y^n \to Y$ of $Y$.
Then
\[
\ACDiff_{W'}(\Gamma) = \ACDiff_{W'}(\nu^* \Gamma +C_Y),
\]
where $C_Y$ denotes the conductor divisor of $\nu$ on $Y^n$.
\end{lem}

\begin{proof}
We first note that the tuple $\mathcal{A}' : =(Y^n, W'', W', j, \pi, g , \nu^*\Gamma+C_Y)$ satisfies the conditions in Setting \ref{setup for AC-Diff}, where $g: W' \to Y^n$ is the morphism induced by $f$, $W'' \subseteq Y^n$ is the reduced image of $g$, and $j$, $\pi$ and $\rho$ are natural morphisms such the the following diagram commutes: 
\[
\xymatrix{
 Y  & Y^n \ar^-{\nu}[l]& \\
W \ar@{^{(}->}^-{i}[u]  & W'' \ar@{^{(}->}^-{j}[u] \ar^-{\rho}[l] & W'. \ar^-{\pi}[l] \ar@/^18pt/[ll]^{\mu} \ar_-{g}[ul] \ar@/_30pt/[ull]_{f}
}\]
We also remark that $\nu^*W=W''$ by the same argument as the proof of \ref{Diff vs Cond2}. 
After shrinking $Y$, we may assume that $K_Y+W+\Gamma$ is $\Q$-Cartier.
Take an integer $m \ge 1$ such that $m \Gamma \in \WSh^{*}(Y)$ and $m(K_Y+W)+m\Gamma$ is Cartier.
 
Let $\alpha_Y : \omega_Y \hookrightarrow \mathcal{K}_Y$ and $\alpha_{Y^n} : \omega_{Y^n} \hookrightarrow \mathcal{K}_{Y^n}$ be inclusions such that a diagram involving $\alpha_{W}$ and $\alpha_{W'}$, similar to \eqref{Tr-transpose diagram}, commutes and let $K_Y$ and $K_{Y^n}$ be the corresponding canonical AC divisors. Then 
$K_{Y^n} = \nu^* K_Y - C_Y$. 
We also take inclusions $\alpha_{W} : \omega_{W} \hookrightarrow \mathcal{K}_{W}$, $\alpha_{W'} : \omega_{W'} \hookrightarrow \mathcal{K}_{W'}$ and $\alpha_{W''} : \omega_{W''} \hookrightarrow \mathcal{K}_{W''}$ such that each two of them satisfy a similar commutativity. 

Let $\mathcal{F} \subseteq \mathcal{K}_Y$ and $\mathcal{G} \subseteq \mathcal{K}_{Y^n}$ denote the submodules corresponding $m(K_Y+W)+m\Gamma$ and $m(K_{Y^n}+W'' +(\nu^*\Gamma +C_Y))= \nu^*(m(K_Y+W)+m \Gamma)$, respectively.
Since $\mathcal{F}$ is invertible, the canonical isomorphism $\theta_{\nu}^{\mathrm{adj}}: \nu^* \mathcal{K}_Y \xrightarrow{\sim} \mathcal{K}_{Y^n}$ induces the isomorphism $\theta_{\nu}^{\mathrm{adj}}: \nu^* \mathcal{F} \xrightarrow{\sim} \mathcal{G}$.
As in the proof of Lemma \ref{AC-Diff vs Cond1}, it suffices to show that the following diagram
\[
\xymatrix{
f^* \mathcal{F} \ar_-{\beta_{\mathcal{A}}}[dr] \ar^{\sim}_-{g^* \theta_\nu^{\mathrm{adj}}}[rr] && g^* \mathcal{G} \ar^-{\beta_{\mathcal{A}'}}[dl] \\
&\mathcal{K}_{W'}&
}\]
commutes, where $\beta_{\mathcal{A}} : = \beta_{\mathcal{A}}(\alpha_Y, \alpha_{W'},m)$ and $\beta_{\mathcal{A}'} : = \beta_{\mathcal{A}'}(\alpha_{Y^n}, \alpha_{W'},m)$.

Take an open subscheme $V \subseteq Y$ such that $V$ is regular, $\Gamma|_V=0$, $U : = V \cap W$ is regular and $U$ contains all generic points of $W$.
Let $\nu': V^n : =\nu^{-1}(V) \xrightarrow{\sim} V$ and $\rho' : U'' : = \rho^{-1}(U) \xrightarrow{\sim} U$ denote the isomorphisms induced by $\nu$ and $\rho$, respectively.
Then the problem can be reduced to showing that the following diagram
\[
\xymatrix{
(\rho')^* (\mathcal{F}|_U) \ar_-{(\rho')^* \beta_{\mathcal{A}}^V}[dr] \ar^{\sim}_-{(\rho')^* (\theta_\nu^{\mathrm{adj} }|_U)}[rr] && \mathcal{G}|_{U''} \ar^-{\beta^{V^n}_{\mathcal{A}'}}[dl] \\
&\mathcal{K}_{U''}&
}\]
commutes, where $\beta_{\mathcal{A}}^V : = \beta_{\mathcal{A}}^V(\alpha_Y, \alpha_{W'},m)$ and $\beta_{\mathcal{A}'}^{V^n} : = \beta_{\mathcal{A}'}^{V^n}(\alpha_{Y^n}, \alpha_{W'},m)$.
By taking $(\rho')^*$-$(\rho')_*$-adjoint, this follows from the commutativity of the following diagram (c.f.~ \eqref{Trace and Residue}).
\[
\xymatrix{
\omega_V(U)|_U \ar_-{\mathrm{Res}_{V/U}}[d] && (\rho')_* (\omega_{V^n}(U'')|_{U''}) \ar_-{\mathrm{Tr}_{\nu'}|_U}[ll] \ar^-{\rho'_* \mathrm{Res}_{V^n/U''}}[d] \\
\omega_U && (\rho')_* \omega_{U''} \ar^-{\mathrm{Tr}_{\rho'}}[ll]
}\]

\end{proof}

\subsection{Existence of a canonical divisor}

In this subsection, we give a sufficient condition for a scheme to admit a canonical AC divisor.

\begin{lem}\label{S1 vs AC}
Let $X$ be an excellent reduced scheme and $\mathcal{F}$ be an $S_1$ coherent sheaf such that $\mathcal{F}_{\eta}$ is an invertible $\sO_{X,\eta}$-module for every generic point $\eta \in X$.
Then there exists an inclusion $\mathcal{F} \hookrightarrow \mathcal{K}_X$.
\end{lem}

\begin{proof}
Let $Q : = \prod_{\eta} \kappa(\eta)$ denote the product of the residue fields $\kappa(\eta)$ of all generic points $\eta \in X$.
Since $X$ is reduced, $\mathcal{K}_X$ is isomorphic to $i_* \widetilde{Q}$, where $i : \Spec Q \to X$ is the natural morphism.

Since $\mathcal{F}$ is invertible at all generic points of $X$, there exists an isomorphism $i^* \mathcal{F} \to \widetilde{Q}$, which induces the adjoint morphism $\alpha: \mathcal{F} \to \mathcal{K}_X$.
Since $\alpha$ is injective at every generic point of $X$ and $\mathcal{F}$ satisfies $(S_1)$, we conclude that $\alpha$ is injective.
\end{proof}

\begin{lem}\label{canonical AC divisor1}
Let $X$ be an excellent reduced $(S_2)$ and $(G_1)$ scheme with a dualizing complex $\omega_X^{\bullet}$.
Let $\delta:X \to \Z$ be the dimension function associated to $\omega_X^{\bullet}$ $($see \cite[Lemma 0AWF]{Sta} for definition$)$.
Then the following conditions are equivalent to each other. 
\begin{enumerate}[label=$(\arabic*)$]
\item $X$ admits a canonical AC divisor associated to $\omega_X^{\bullet}$.
\item The support of the canonical sheaf $\omega_X$ coincides with $X$.
\item $\delta(\eta)$ is constant for all generic points $\eta$ of $X$.
\end{enumerate}
\end{lem}

\begin{proof}
Since $\omega_X$ satisfies $(S_2)$ and the support of $\omega_X$ is the union of the irreducible components of maximal dimension with respect to $\delta$ (see \cite[Lemma 0AWK]{Sta}), the assertion follows from Lemma \ref{S1 vs AC}.
\end{proof}

\begin{defn}\label{biequidim}
A topological space $X$ of finite Krull dimension is \emph{biequidimensional} if all maximal chains of irreducible closed subsets of $X$ have the same length.
\end{defn}

\begin{eg}[\textup{\cite[Lemma 2.4]{Hei}}]\label{biequidim vs equidim}
If $X=\Spec R$ is the spectrum of a Noetherian local ring $R$, then $X$ is biequidimensional if and only if $X$ is equidimensional and catenary.
\end{eg}

\begin{lem}\label{canonical AC divisor2}
Let $X$ be an excellent reduced $(S_2)$ and $(G_1)$ scheme with a dualizing complex $\omega_X^{\bullet}$.
Then $X$ admits a canonical AC divisor associated to $\omega_X^{\bullet}$ if one of the following conditions hold.
\begin{enumerate}[label=\textup{(\roman*)}]
\item There exists a finite morphism $f: X \to Y$ to an excellent reduced $(S_2)$ and $(G_1)$ scheme $Y$ with the following conditions:
\begin{enumerate}[label=\textup{(\alph*)}]
\item $Y$ admits a dualizing complex $\omega_Y^{\bullet}$ such that $f^! \omega_Y^{\bullet} \cong \omega_X^{\bullet}$, 
\item $Y$ admits a canonical AC divisor associated to $\omega_Y^{\bullet}$, and 
\item the codimension of the point $f(\eta) \in Y$ is constant for all generic points $\eta$ of $X$.
\end{enumerate}
\item $X$ is irreducible.
\item $X$ is connected and biequidimensional.
\end{enumerate}
\end{lem}

\begin{proof}
By Lemma \ref{canonical AC divisor1}, it is enough to show that $\delta(\eta)$ is constant for all generic points $\eta$ of $X$.
In the case (ii), this is obvious.
In the case (iii), it follows from \cite[Lemma 02IA]{Sta}.
In the case (i), let $\delta': Y \to \Z$ denote the dimension function associated to $\omega_Y^{\bullet}$.
Then it follows from \cite[Lemma 0AX1]{Sta} that $\delta = \delta' \circ f$.
Since $Y$ admits a canonical AC divisor, it follows from Lemma \ref{canonical AC divisor1} that $\delta'(y)=\delta'(y')$ for any points $y, y' \in Y$ with same codimension.
Thus the assertion follows again from Lemma \ref{canonical AC divisor1}.
\end{proof}

We next give a sufficient condition for the map
\[
\WSh^*(X) \to \WSh(X) \to \WSh(X)/\sim
\]
to be surjective, where $\sim$ denotes the linear equivalence of AC divisors.

\begin{lem}\label{Moving}
Let $(\Lambda, \m, k)$ be a Noetherian local ring with $k$ infinite, $A$ be a Noetherian $\Lambda$-algebra and $X$ be a quasi-projective $A$-scheme.
Suppose that $\Sigma \subseteq X$ is a finite subset and $D$ is an AC divisor which is Cartier at any points of $\Sigma$.
Then there exists an AC divisor $D'$ linearly equivalent to $D$ such that $\Sigma \cap \Supp D'=\emptyset$.
In particular, the map 
\[
\WSh^*(X) \to \WSh(X) \to \WSh(X)/\sim
\]
is surjective.
\end{lem}

\begin{proof}
Let $\mathcal{F} \subseteq \mathcal{K}_X$ be a submodule corresponding to $D$.
After twisting $\mathcal{F}$ by an ample line bundle, we may assume that $\mathcal{F}$ is globally generated.
We set $M : = H^0(X, \mathcal{F})$ and $N_x : = \mathrm{Ker} (M \to \mathcal{F}_x \otimes \kappa(x)) \subseteq M$ for every point $x \in \Sigma$.
The global generation of $\mathcal{F}$ yields that $N_x \neq M$.
Taking into account that $k$ is infinite, we have 
\[
\bigcup_{x \in \Sigma} N_x \neq M.
\]
Take an element $r \in M \setminus \bigcup_{x\in \Sigma}N_x$. 
Since $\mathcal{F}$ is Cartier at any $x \in \Sigma$, the support of the AC divisor $D' : = D + \Div_X(r)$ does not contain $x$, as desired.
\end{proof}

\subsection{Grotherndieck-Serre duality}\label{duality}

In this subsection, we use the same convention as in Subsection \ref{Subsection AC-Diff}. 
First we recall the notation of Grotherndieck-Serre duality. 
The reader is referred to \cite[Subsection 2.3]{BST}, \cite[VII Corollary 3.4]{RD} and \cite[Section 3]{Con} for details.

Given two morphisms $f: Y \to X$ and $g: Z \to Y$, there exists a canonical isomorphism of functors
\[
c_{g,f}: (f g)^{!} \xrightarrow{\sim} g^! f^! ,
\]
which is compatible with triple composition (see \cite[VII Corollary 3.4 (VAR1)]{RD}).
When $f$ is the structure morphism $\pi_Y : Y \to S$ of the $S$-scheme $Y$, we write 
\[
c_g : = c_{g, \pi_Y }(\omega_S^{\bullet}): \omega_Z^{\bullet} \xrightarrow{\sim} f^{!}\omega_Y^{\bullet}.
\]
For an open immersion $i: U \hookrightarrow X$, we have an isomorphism $e_i: i^! \xrightarrow{\sim} i^*$ and write
\[
\gamma_i:= e_i \circ c_i : \omega_U^{\bullet} \xrightarrow{\sim} \omega_X^{\bullet}|_U.
\]
Since $e_i$ is compatible with composition (see \cite[VII Corollary 3.4 (VAR3)]{RD}), so is $\gamma_i$, that is, if $j : V \hookrightarrow U$ is an open immersion, then the following diagram commutes.
\begin{equation}\label{open composition diagram}
\vcenter{
\xymatrix{
\omega_V^{\bullet} \ar^-{\gamma_j}[rr] \ar_-{\gamma_{i \circ j}}[dr] & & \omega_U^{\bullet}|_V \ar^-{\gamma_i|_V}[dl]\\
& \omega_X^{\bullet}|_V&
}}
\end{equation}
For a proper morphism $f: Y \to X$, the trace map $Rf_* f^! \to \mathrm{id}$ (see \cite[Subsection 3.4]{Con}) induces the morphism
\[
\xymatrix{
\mathrm{Tr}_f^{\bullet}: Rf_* \omega_{Y}^{\bullet} \ar[r]^-{\sim}_-{Rf_*c_f}& Rf_* f^! \omega_X^{\bullet} \ar[r] & \omega_X^{\bullet}, 
}\]
which we also refer to as the trace map of $f$. 
Then the following properties hold (see \cite[VII Corollary 3.4 (TRA1))]{RD}, \cite[VII Corollary 3.4 (TRA4))]{RD} and \cite[VII Corollary 3.4 (c)]{RD}). 
\begin{enumerate}
\item Trace maps are compatible with composition, that is, if $f: Y \to X$ and $g: Z \to Y$ are proper morphisms, then the following diagram commutes.
\begin{equation}\label{Trace composition diagram}
\vcenter{
\xymatrix{
R (fg)_* \omega_Z^{\bullet} \ar^-{\mathrm{Tr}_{f \circ g}^{\bullet}}[rr] \ar_-{\rotatebox{90}{$\sim$}}[d] && \omega_X^{\bullet} \\
Rf_* R g_* \omega_Z^{\bullet} \ar_-{Rf_* \mathrm{Tr}_g^{\bullet}}[rr] && Rf_* \omega_Y^{\bullet} \ar_-{\mathrm{Tr}_f^{\bullet}}[u]
}}
\end{equation}

\item Trace maps are compatible with the base change by an open immersion, that is, if 
\[
\xymatrix{
Y \ar^-{f}[r] & X \\
V \ar@{^{(}->}[u]^-{j} \ar_-{g}[r] & U \ar@{^{(}->}[u]_-{i}
}\]
is a cartesian diagram with $f$ proper and $i$ an open immersion, then the diagram
\begin{equation}\label{Trace open diagram}
\vcenter{
\xymatrix{
i^* Rf_* \omega_Y^{\bullet} \ar_-{\rotatebox{90}{$\sim$}}[d] \ar^-{i^* \mathrm{Tr}_f^{\bullet}}[rrr] &&& i^* \omega_X^{\bullet} \\
Rg_* j^*\omega_Y^{\bullet}   \ar^{\sim}_-{(Rg_*\gamma_j)^{-1}}[rr] && Rg_* \omega_{V}^{\bullet} \ar_-{\mathrm{Tr}_g^{\bullet}}[r] & \omega_U^{\bullet} \ar_-{\rotatebox{90}{$\sim$}}^-{\gamma_i}[u]
}}
\end{equation}
commutes.

\item Given a proper morphism $f: Y \to X$ and a coherent sheaf $\mathcal{F}$ on $Y$, the morphism
\[
Rf_*R\mathcal{H}om_Y( \mathcal{F}, \omega_Y^{\bullet}) \to  R\mathcal{H}om_X(Rf_*\mathcal{F}, Rf_* \omega_Y^{\bullet}) \xrightarrow{\mathrm{Tr}_f^{\bullet}} R\mathcal{H}om_X(Rf_*\mathcal{F},  \omega_X^{\bullet})
\]
induced by the trace map $\mathrm{Tr}_f^{\bullet}$ of $f$ is an isomorphism.
\end{enumerate}

Let $f: Y \to X$ be a finite surjective morphism.
Since we have the isomorphisms 
\[
\omega_Y^{\bullet} \cong f^! \omega_X^{\bullet} \cong R\mathcal{H}om_X(f_* \sO_Y, \omega_X^{\bullet}),
\]
if the canonical sheaf $\omega_X$ associated to $\omega_X^{\bullet}$ is the $-d$-th cohomology $h^{-d}(\omega_X^{\bullet})$, then the canonical sheaf $\omega_Y$ associated to $\omega_Y^{\bullet}$ is also the $-d$-th cohomology $h^{-d}(\omega_Y^{\bullet})$. 
Therefore, the trace map $\mathrm{Tr}_f^{\bullet}$ induces the morphism of sheaves 
\[
\mathrm{Tr}_f : f_* \omega_Y \to \omega_X, 
\]
which we also refer to as the \emph{trace map} of $f$.

Let $X$ be a Cohen-Macaulay scheme with $\Supp \omega_X= X$ and $Z \subseteq X$ be an effective Cartier divisor, that is, a closed subscheme whose ideal sheaf is invertible.
Then the exact sequence
\[
0 \to \sO_X(-Z) \to \sO_X \to i_* \sO_Z \to 0
\]
induces a morphism
\[
\omega_X^{\bullet}(Z) \cong R\mathcal{H}om_X(\sO_X(-Z),  \omega_X^{\bullet}) \to
R\mathcal{H}om_X(i_* \sO_Z, \omega_X^{\bullet})[1] \cong 
i_* \omega_Z^{\bullet}[1],
\]
where $i: Z \hookrightarrow X$ denotes the closed immersion.
Noting that $\omega_X^{\bullet} \cong \omega_X [d]$ and $\omega_Z^{\bullet} \cong \omega_Z [d-1]$ for some integer $d$, we have a morphism of sheaves  
\[
\omega_X(Z) \to i_* \omega_Z.
\]
Let $\mathrm{Res}_{X/Z} : \omega_X(Z)|_Z \to \omega_Z$ be the adjoint morphism and we call it the \emph{Poincar\'{e} residue map}.
Making use of the diagrams \eqref{open composition diagram} and \eqref{Trace composition diagram}, we have the following properties of Poincar\'{e} residue maps. 
\begin{enumerate}
\item Poincar\'{e} residue maps are compatible with open immersions, that is, if 
\[\xymatrix{
X^{\circ} \ar@{^{(}->}^-{u}[r] & X \\
Z^{\circ} \ar@{^{(}->}^-{v}[r] \ar@{^{(}->}^-{j}[u] & Z \ar@{^{(}->}^-{i}[u]
}\]
is a commutative diagram where  $u$ and $v$ are open immersions and $j$ is a closed immersion, then the following diagram commutes.
\begin{equation}\label{Res open}
\vcenter{
\xymatrix{
v^*(\omega_X(Z)|_Z) \ar^-{\sim}[r] \ar^-{v^* \mathrm{Res}_{X/Z}}[d] & (u^*\omega_X(Z))|_{Z^\circ} \ar^-{\gamma_u^{-1}(Z|_{X^{\circ}})|_{Z^{\circ}}}[rr] & & \omega_{X^{\circ}}(Z^{\circ})|_{Z^{\circ}} \ar^-{\mathrm{Res}_{X^{\circ}/Z^{\circ}}}[d] \\
v^* \omega_Z \ar^-{\gamma_v^{-1}}[rrr] &  & & \omega_{Z^{\circ}}
}}
\end{equation}

\item Poincar\'{e} residue maps are compatible with trace maps, that is, if 
\[
\xymatrix{
Y \ar^-{f}[r]  & X \\
W \ar^-{g}[r] \ar@{^{(}->}^-{j}[u] & Z \ar@{^{(}->}_-{i}[u]
}
\]
is a commutative diagram where $Y$ is another Cohen-Macaulay scheme, $W \subseteq Y$ is an effective Cartier divisor, and $f: Y \to X$ and $g: W \to Z$ are finite surjective morphisms, 
then the following diagram commutes. 
\begin{equation}\label{Trace and Residue}
\vcenter{
\xymatrix{
(f_* \omega_{Y}(W))|_{Z} \ar^-{\mathrm{Tr}_{f}|_Z}[d] \ar^-{\mathrm{nat}}[r]  &  g_* (\omega_{Y}(W)|_{W}) \ar^-{g_* \mathrm{Res}_{Y/W}}[rr]  & & g_* \omega_{W} \ar^-{\mathrm{Tr}_{g}}[d] \\
\omega_X(Z)|_Z  \ar^-{\mathrm{Res}_{X/Z}}[rrr] &&& \omega_Z 
}}\end{equation}
\end{enumerate}

\begin{lem}\label{residue isom}
Let $X$ be a Cohen-Macaulay scheme with $\Supp \omega_X=X$ and $Z \subseteq X$ be an effective Cartier divisor.
Then the Poincar\'{e} residue map 
\[
\mathrm{Res}_{X/Z} : \omega_X(Z)|_Z \to \omega_Z
\]
is surjective.
Moreover, if $Z$ is reduced, then it is an isomorphism.
\end{lem}

\begin{proof}
The distinguished triangle 
\[
R\mathcal{H}om_X(i_* \sO_Z, \omega_X^{\bullet}) \to R\mathcal{H}om_X(\sO_X, \omega_X^{\bullet}) \to R\mathcal{H}om_X(\sO_X(-Z), \omega_X^{\bullet}) \xrightarrow{+1}
\]
yields the exact sequence
\[
\omega_X(Z) \to i_* \omega_Z \to \mathcal{E}xt_X^1(\sO_X, \omega_X)=0.
\]
Therefore, the morphism $\omega_X(Z) \to i_*\omega_Z$ is surjective, which proves the first assertion.

For the second assertion, it suffices to show that the residue map $\mathrm{Res}_{X/Z}$ is injective at any generic points of $Z$, because $\omega_X(Z)|_Z$ satisfies $(S_1)$. 
Since $Z$ is reduced, after shrinking $X$ and $Z$, we may assume that $X$ and $Z$ are both regular.
Then $\omega_X(Z)$ and $\omega_Z$ are invertible sheaves and every surjective morphism of invertible sheaves is an isomorphism.  
\end{proof}


\end{document}